\documentclass[11pt]{article}
\usepackage[utf8]{inputenc}
\usepackage[T1]{fontenc}
\usepackage{macros}
\usepackage{amsmath}
\usepackage{tikz-cd}
\usepackage{tikz}
\usepackage{amsthm}
\usepackage{amssymb}
\usepackage{stmaryrd}
\usepackage{amsfonts}
\usepackage{graphicx}
\usepackage{caption}
\usepackage{subcaption}
\usepackage{xfrac, faktor}
\usepackage{xspace}
\usepackage[all]{xy}
\usepackage{txfonts}
\usepackage{csquotes}
\usepackage{mathtools}
\usepackage{thmtools}
\usepackage{tdsfrmath}
\usepackage{xfrac}
\usepackage[shortlabels]{enumitem}
\usepackage{geometry}
\usepackage{arcs}
\usepackage{txfonts}
\usepackage{comment}
\usepackage{float}
\usepackage[title]{appendix}
\usepackage[
backend=biber,
style=alphabetic,
sorting=ynt
]{biblatex}
\usepackage{hyperref}
\newcommand{\footremember}[2]{%
    \footnote{#2}
    \newcounter{#1}
    \setcounter{#1}{\value{footnote}}%
}

\newtheorem{theorem}{Theorem}[section]
\newtheorem{lemma}[theorem]{Lemma}

\newtheorem{proposition}[theorem]{Proposition}

\theoremstyle{remark}
\newtheorem{remark}[theorem]{Remark}
\newtheorem{notation}[theorem]{Notation}

\theoremstyle{definition}
\newtheorem{definition}[theorem]{Definition}
\newtheorem{example}[theorem]{Example}

\newtheorem*{thm*}{Theorem}
\newtheorem*{cor*}{Corollary}
\newtheorem*{conj*}{Conjecture}

\theoremstyle{remark}
\newtheorem{step}{Step}

\newcounter{notes}%

\DeclareFontFamily{OMX}{yhex}{}
\DeclareFontShape{OMX}{yhex}{m}{n}{<->yhcmex10}{}
\DeclareSymbolFont{yhlargesymbols}{OMX}{yhex}{m}{n}
\DeclareMathAccent{\wip}{\mathord}{yhlargesymbols}{"F3}

\title{Polyhedral realisations of finite arc complexes using strip deformations}
\author{%
  François Guéritaud\footremember{stras}{Université de Strasbourg}%
  \and Pallavi Panda\footremember{uspn}{Université Paris 13}
  }

\date{}

\begin{document}
\maketitle

\abstract{ 
We study infinitesimal deformations of complete hyperbolic surfaces with boundary and with ideal vertices, possibly decorated with horoballs.
``Admissible'' deformations are the ones that pull all horoballs apart; they form a convex cone of deformations.
We describe this cone in terms of the arc complex of the surface: specifically, this paper focuses on the surfaces for which that complex is finite.
Those surfaces form four families: (ideal) polygons, once-punctured polygons, one-holed polygons (or ``crowns''), and M\"obius strips with spikes.
In each case, we describe a natural simplicial decomposition of the projectivized admissible cone and of each of its faces, realizing them as appropriate arc complexes.
}

\section{Introduction}

A hyperbolic surface $\Pi$ comes with a deformation space $\mathfrak{D}(\Pi)$, sometimes called its Fricke-Klein space (or Teichm\"uller space in the context of conformal geometry), the space of hyperbolic metrics on $\Pi$ up to isotopy. 
This space $\mathfrak{D}(\Pi)$ has been known since at least the time of Koebe and Poincaré to be diffeomorphic to a smooth open ball, whose dimension depends on the topological type of $\Pi$. 
It carries a wealth of geometric structures, both differential and large-scale, that have been studied by many authors.

We may also talk about \emph{infinitesimal} deformations of a hyperbolic metric on $\Pi$. 
While falling squarely in the differential study of $\mathfrak{D}(\Pi)$ (see e.g.\ Kerckhoff's solution to the Nielsen realization problem~\cite{kerck}), this idea has also proven crucial, for instance, in the study of $3$-manifolds with complete flat Lorentzian metrics~\cite{glm}, \cite{dgk}. 
The hyperbolic surfaces $\Pi$ of interest in that context are typically compact with geodesic boundary, and the deformations that give rise to ``good'' Lorentzian metrics are the ones that \emph{infinitesimally lengthen} every geodesic loop in~$\Pi$, in a uniform sense. 
These are traditionally called \emph{admissible} deformations;  they form a convex cone in any fiber of $\mathrm{T} \mathfrak{D}(\Pi)$.

The main result of \cite{dgk} is a realization theorem of admissible deformations in terms of the \emph{arc complex} of $\Pi$. This clique simplicial complex, of dimension one less than $\mathfrak{D}(\Pi)$, first introduced by Harvey \cite{Harvey}, is naturally defined in terms of the topology of $\Pi$. In general, it is an infinite complex without any nice piece-wise linear manifold structure. Nonetheless, Hatcher \cite{hatcher} proved that it is always contractible. Furthermore, Harer \cite{harer} proved that there exists an open dense subset, called the pruned arc complex, that is homeomorphic to an open ball. The latter is used in \cite{dgk} to parametrize admissible deformations up to scaling.


\smallskip

In this paper we study those hyperbolic surfaces $\Pi$ that give rise to \emph{finite} arc complexes, and interpret these complexes precisely in terms of a generalized notion of admissible deformation. 
In order to capture many finite cousins of the associahedron, we must allow the hyperbolic surface $\Pi$ to have so-called ``spikes'', i.e.\ regions isometric to a corner of an ideal hyperbolic polygon.
An ideal polygon in $\HP$ is itself a possible (and important) case of such a surface~$\Pi$, hence our somewhat unusual choice of letter.
We will also ``decorate'' some of these spikes with \emph{horoballs}: the horoball data is seen as part and parcel of the ``metric'' $m$ on~$\Pi$. This was first introduced by Penner in his work on Decorated Teichmüller Theory \cite{Pennerbordered},\cite{Pennerpunc}. 
A deformation of $m$, or vector of $\mathrm{T}_m \mathfrak{D}(\Pi)$, will be called admissible if it lengthens not only every loop, but also every inter-horoball distance. 
(This will be our working definition of admissibility, mostly ignoring here the connections to Lorentzian geometry.)
Our main theorem is:
\begin{theorem}\label{thm:mainA}
Let $\Pi$ be a (decorated) surface with finite arc complex, and $m\in \mathfrak{D}(\Pi)$ a hyperbolic metric. 
There is a natural parametrization of the closed polyhedron of (weakly) admissible deformations 
$\padm$
in the positive projectivization  $\mathbb{P}^+ \mathrm{T}_m\mathfrak{D}(\Pi)$, by an appropriate arc complex in $\Pi$.
Moreover, there is an explicit description of the faces of $\padm$ in terms both of lengthening properties, and of arc subcomplexes.
\end{theorem}

In order to state the result precisely (Theorem~\ref{thm:mainB} below), we will need to make a number of definitions in Section~\ref{sec:defs}.
Nevertheless, we can describe right away the class of spiked surfaces $\Pi$ (ignoring decorations) to which Theorem~\ref{thm:mainA} applies: these are 

(1) the spiked disks (ideal polygons), 

(2) once-punctured spiked disks (once-punctured polygons,  

(3) once-holed spiked disks (``crowns''), and 

(4) spiked M\"obius strips. 

\noindent In particular, their fundamental groups are all either trivial or cyclic.
Among complete connected hyperbolic surfaces with nonempty, geodesic boundary and finite, nonzero area, these are the only examples with finite arc complexes that have spikes. In \cite{ppstrip} and \cite{ppdmst}, it was shown that the interior of the arc complex parametrises the projectivised admissible cone. In \cite{ppballs} and \cite{ppstrong}, it was shown that the full arc complexes of these four types of decorated surfaces are closed balls. Consequently, the projectivised strip map was shown to be a homeomorphism between the full arc complex and $\padm$.

If we include spikeless surfaces then there are two more cases with finite arc complexes: the three-holed sphere and two-holed projective plane. These two surfaces have 
$\dim \overline{\Lambda(m)}^+=2$
(or less if we allow some of the holes to be converted to punctures) and are easy to treat separately, as was done in~\cite[\S6]{dgk}.

Note that our surfaces may have punctures (also called ``cusps'': regions isometric to the quotient of a horoball by a parabolic isometry), as exemplified by~(2) in the above list; but we do not consider any ``horoball decorations'' on the cusps themselves, only on the spikes.

For surfaces with infinite arc complexes, a full counterpart of Theorem~\ref{thm:mainA} does exist but it is quite a bit more delicate to state (and prove). 
We will deal with it in future work.


\subsection*{Plan of the paper}
In Section~\ref{sec:defs} we define most objects needed to make Theorem~\ref{thm:mainA} precise.
In Section~\ref{sec:mainresult} we state and prove the main result, Theorem~\ref{thm:mainB}.
The remaining sections explore consequences and special cases.
In Section~\ref{sec:walls} we characterize the facets (top-dimensional faces) of the admissible cone.
In Section~\ref{sec:fully} we describe the vertices of the admissible cone, in the case where it is a properly convex polytope (i.e.\ when all spikes of $\Pi$ are decorated).
In Appendix \ref{sec:app1} we collect a few facts used in the proof of the main result; these are in large part known from the literature but we present them in a unified way useful to our purpose. Finally, in Appendix \ref{sec:app2}, we illustrate Theorem~\ref{thm:mainB} and Proposition~\ref{prop:dimofmac} in the case of a fully decorated 2-crown $\dholed 2$. 


\section{Definitions} \label{sec:defs}

\subsection{Topological type and dimension} \label{sec:toptype}
For $g\geq 0$, let $\mathsf{S}_g$ be the connected sum of $g$ copies of the torus (or standard surface of genus~$g$), with $\mathsf{S}_0$ the $2$-sphere.
For $h\geq 1$, let also $\mathsf{N}_h$ be the connected sum of $h$ copies of the projective plane.
A complete hyperbolic surface $\Pi$ with nonempty geodesic boundary and finite (positive) area can always be obtained, up to homeomorphism, from some (unique) surface $\mathsf{S}_g$ or $\mathsf{N}_h$, by removing the interiors of disjoint closed disks $D_1, \dots, D_\ell$, and then attaching to each boundary loop $\partial D_i$ either a punctured disk $\mathbb{S}^1\times [0,+\infty)$, or an annulus $\mathbb{S}^1\times [0,1]$, or an $n_i$-\emph{crown} which is an annulus minus $n_i\geq 1$ points of $\mathbb{S}^1\times \{1\}$.
(Gluings happen along $\mathbb{S}^1\times \{0\}$.)

The ``missing'' point at infinity of a punctured disk $\mathbb{S}^1\times [0,+\infty)$ gives rise to a so-called puncture of $\Pi$, surrounded by a cusp region. 
The annulus boundaries $\mathbb{S}^1\times \{1\}$ give rise to the compact (geodesic) boundary components of $\Pi$, while the $n_i$ open segments in the boundary of an $n_i$-crown give rise to $n_i$ noncompact geodesic boundary components of $\Pi$: metrically, two consecutive of these $n_i$ segments are infinite geodesics that become asymptotically close near their common end.
(For $n_i=1$ it could be the two ends of one geodesic.)

Finally, when $n_i>0$, we may decorate $r_i$ of the $n_i$ spikes at the $i$-th hole with horoballs.
If $\ell' \leq \ell$ is the number of punctures, then the deformation space of such a surface $\Pi$ has dimension
$$N_\Pi=\left \{ \begin{array}{ll} 
6g-6+2\ell' + 3(\ell-\ell') + \sum_i n_i+r_i, & \text{if $\Pi$ is orientable} \\ 
3h-6+2\ell' + 3(\ell-\ell') + \sum_i n_i+r_i, & \text{otherwise} \end{array} \right .$$
which is non-negative. 
Note that our assumption of nonempty boundary means\ $\ell>\ell'$; in particular $\ell \geq 1$.
In fact our initial assumption that $\Pi$ be a hyperbolic surface with \emph{positive} area is slightly \emph{stronger} than $N\geq 0$ and $\ell>\ell'$:

\noindent --- When $(g,\ell, \ell')=(0,1,0)$ then $\Pi$ is an $n_1$-gon and $N_\Pi=n_1+r_1-3$. 
The demand of positive area requires $n_1\geq 3$, slightly stronger than $N_\Pi \geq 0$.

\noindent --- When $(g,\ell, \ell')=(0,2,0)$, then $\Pi$ is an annulus with $n_1+n_2$ spikes and $N_\Pi = (n_1+n_2)+(r_1+r_2)$: positive area requires $n_1+n_2 \geq 1$, stronger than $N\geq 0$.

\noindent --- When $(h, \ell, \ell')=(1,1,0)$, similarly $\Pi$ is an $n_1$-spiked M\"obius strip and $N_\Pi = n_1+r_1$: positive area requires $n_1 \geq 1$, stronger than $N_\Pi \geq 0$.

\noindent --- In all other cases, one may check that $N_\Pi\geq 0$ already implies that the surface of the corresponding topological type can be realized as a hyperbolic surface with positive area.
In particular, when $(g, \ell, \ell')=(0,2,1)$, then $\Pi$ is a punctured $n_1$-gon and $N_\Pi = n_1+r_1-1$: positive area requires $n_1 \geq 1$, \emph{equivalent} to $N_\Pi \geq 0$ (recall $0\leq r_1\leq n_1$).

\subsection{Notation} \label{sec:notation}

The four special cases just enumerated are in fact, essentially, the ones Theorem~\ref{thm:mainA} will apply to (i.e.\ the surfaces with finite arc complexes). 
Since we will mostly focus on these surfaces, we introduce special notation for them.
Each notation in the table below represents a finite family of topological types, depending on how the $r$ decorations are distributed among the $n$ spikes.
The rightmost column can be ignored for the moment.

$$	\begin{array}{|c|c|c|c|c|c|}
		\hline
	\text{Notation} &\text{\# spikes} &\text{\# decorations} &  \text{Family of surfaces}& N_\Pi & \# \text{ simple } \be\\		\hline
	\pdep nr &n\geq 3 & r\geq 0 &	 \text{ Ideal $n$-gon}&n+r-3 &\dfrac{r(r-1)}{2}\\[0.8em]
	\puncp nr &n\geq 1 & r\geq 0 &	 \text{ Punctured $n$-gon}&n+r-1& r^2 \\[0.5em]
	\pholed nr &n\geq 1  & r\geq 0 & \text{$n$-spiked crown}& n+r & r^2 + 1\\[0.5em]
	\pmob nr & n\geq 1  & r\geq 0 &	 \text{$n$-spiked M\"obius strip}& n+r & \dfrac{r(r+1)}{2}+r^2+1\\[0.5em]
	\hline
	\end{array}$$

(In keeping with the previous section, a \emph{crown} is an annulus with $0$ spikes on one of its boundary components.)



It is straightforward to check from the classification above that these are exactly the hyperbolic surfaces with spikes that do \emph{not} contain any two-sided simple loop other than the boundary.
By contrast, any surface that does  possess such a loop must have infinite mapping class group (because of Dehn twist), hence infinitely many arcs.

For\begin{figure}[h!]
    \centering
    \includegraphics[width=0.7\linewidth]{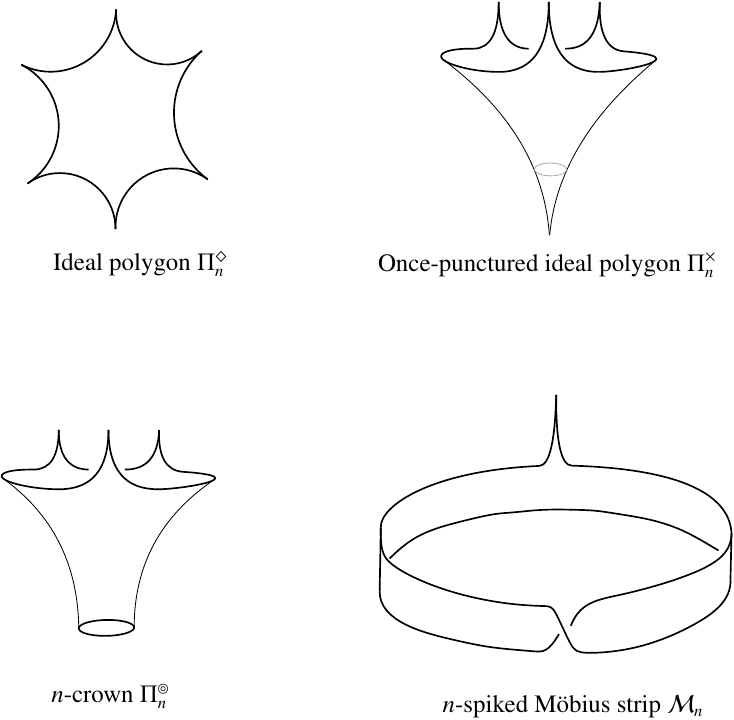}
    \caption{\emph{The four families of surfaces with finite arc complexes. The index $n$ refers to the number of spikes: here $\Pi^\Diamond_6, \Pi^\times_3, \Pi^\circledcirc_3~, \mathcal{M}_1$.}}
    \label{fig:4fam}
\end{figure} concreteness, we now describe the four types of surfaces directly (see Figure~\ref{fig:4fam}).

\paragraph{Ideal polygons.}
An ideal $n$-gon, denoted by $\ip n$, is the convex hull in $\HP$ of $n\geq 3$ points in the boundary of $\HP$.
This surface is homeomorphic to a disk with $n$ points removed from its boundary.
The infinite hyperbolic geodesic joining two consecutive spikes is called an \emph{edge}. 
\paragraph{Once-punctured polygons.}A once-punctured $n$-gon, denoted by $\Pi_n^\times$, with ($n\geq 1$), is obtained from an ideal ($n+2$)-gon by identifying two consecutive edges via a parabolic isometry of $\HP$ fixing their common endpoint.
The resulting surface is homeomorphic to a punctured disk with $n$ points removed from its boundary. 

\paragraph{Crowns.}	The surface obtained by removing $n\geq 1$ points from one boundary component of an annulus $\s1 \times [0,1]$, is called a \emph{crown} and is denoted by $\crown n$.
Its boundary consists of one simple closed curve and $n$ open intervals.
Its fundamental group $\fg{\crown n} \simeq \Z$ is generated by a loop $b$ in the free homotopy class of the simple closed curve.
We can specify a hyperbolic structure on $\holed n$ in the following way:
let $\ga\in \psl$ be a hyperbolic element whose axis is a bi-infinite geodesic, denoted by $l$.
It divides the boundary circle $\HPb$ into two open intervals.
Let $x_1$ be a point in one of them, and choose $(n-1)$ distinct points $x_2,\ldots,x_n$ on the same interval between $x_1$ and its image $\ga\cdot x_1$. 
The points $\gamma^k \cdot x_i$, for $(k,i)\in \mathbb{Z} \times \llbracket 1,n \rrbracket$, all lie on the same side of $l$; their closed convex hull is invariant under $\ang{\gamma}$, and its quotient is a complete finite-area hyperbolic surface with geodesic boundary, homeomorphic to $\holed n$.
If $\rho:\fg {\holed n}\longrightarrow \psl$ is the holonomy representation, then $\rho(b)=\gamma$.
The images of the ideal points $x_1,\ldots,x_n$ in the quotient are called \emph{spikes}.

\paragraph{M\"obius strip with spikes.}
The non-orientable surface obtained by removing $n\geq 1$ points from the boundary of a M\"obius strip, is called a {spiked M\"obius strip}. 
Its orientation double cover is an annulus minus $n$ points on each boundary component. 
One can put a hyperbolic structure on this surface in the following way: 
consider an ideal $n+2$-gon with marked vertices $x_1,\ldots,x_{n+2}$.
Take any two distinct edges $e,e'$ of the polygon.  
Let $l$ be a geodesic that intersects $e$ and $e'$ such that the angles of intersection sum to $\pi$, and let $d$ be the distance between $l\cap e$ and $l\cap e'$. 
Let $h\in \pgl$ be the glide reflection along $l$ with translation length $d$.
Then the quotient $\ip {n+2}/\sim $, where for every $z\in e$ we set $z \sim h\cdot z$, is a metrically complete finite-area hyperbolic surface.


\subsection{Admissible deformations}
\subsubsection{Horoball connections} 

\begin{definition}
Consider a surface $\Pi$ as in the table above, with $r$ of its $n$ spikes decorated by horoballs. A bi-infinite geodesic joining two decorated spikes is called a \emph{horoball connection}.
The \emph{length} of a horoball connection is defined as the hyperbolic length of the geodesic segment subtended by the two horoballs.
A horoball connection is called \emph{simple} if it has no self-intersection. A simple horoball connection is called \emph{separating} if it decomposes the surface into two nonempty subsurfaces.
\end{definition}

\begin{remark}
In the surfaces $\pdep nr, \puncp nr, \pholed nr, \pmob nr$ the number of simple horoball connections is finite;
but except in $\pdep nr$ the number of nonsimple ones is infinite (for $r>0$).
\end{remark}

\begin{definition} \label{def:maximalHC}
     For $\Pi \in \puncp nr, \pholed nr, \pmob nr$, a separating horoball connection is called \emph{maximal} if one of the two subsurfaces is homeomorphic to $\pdep{n+1}{r+1}$. 
     (For $\Pi \in \pdep nr$ there is no corresponding notion.)
\end{definition}

\begin{remark} \label{rem:sep}
Some simple horoball connections can be sides of $\Pi$; we count them as separating.
For $\Pi \in \pdep nr, \puncp nr, \pholed nr$, all simple horoball connections are separating.
For $\Pi \in \pmob nr$, there exist both separating and non-separating simple horoball connections.
\end{remark}

\subsubsection{The admissible cone}
Let $\mathcal{H}(\Pi)$ be the set of all horoball connections and closed loops of the surface
$\Pi$.
Then we can define the following smooth positive function for every $\be\in \mathcal{H}(\Pi)$:
\[
\begin{array}{cccl}
	l_\be:&\tei \Pi & \longrightarrow & \R_{>0} \\
	&m & \longmapsto & \text{length of $l_\be$ with respect to $m$.}
\end{array}
\] 
\begin{definition}
    Given a hyperbolic metric $m$ on $\Pi$, a vector in the tangent space $\mathrm{T}_m \mathfrak{D}(\Pi)$ is called an \emph{infinitesimal deformation} of $\Pi$.
\end{definition}
\begin{definition}
    An infinitesimal deformation $v$ is said to be \emph{admissible} if there exists a constant $K>0$ such that for all $\be\in \mathcal{H}(\Pi)$, 
\[ \mathrm{d}l_\be (v) \geq K l_\be .\] 
(The condition is vacuously true if $\mathcal{H}(\Pi) = \varnothing$, e.g.\ for undecorated polygons $\Pi=\pdep n0$.)
\end{definition}
\begin{definition}
   The \emph{admissible cone}, denoted by $\adm$, of a metric $m\in\tei {\Pi}$ on a partially decorated surface $\Pi$, is defined to be the set of all admissible deformations of $m$. 
\end{definition}

The admissible cone is open and convex in $\mathrm{T}_m \mathfrak{D}(\Pi)$.
In the case of a decorated polygon $\Pi\in \pdep nr$, there are only finitely many horoball connections, say $\be_1,\ldots, \be_{p(r)}$.
Thus the admissible cone is the intersection of finitely many open half-spaces:
\[\adm =\bigcap\limits_{i=1}^{p(r)} \, \{dl_{\be_i}>0\}. \]

\begin{definition}
    A \emph{weakly admissible} infinitesimal deformation is by definition an element of the \emph{closure} of the admissible cone.
\end{definition}

\begin{notation}
   For a horoball connection or nontrivial closed curve $\be$ in $\Pi$, we define $$H_\be(m):=\{ u\in \mathrm{T}_m \mathfrak{D}(\Pi) \mid \mathrm{d}l_{\be}(m)(u)=0 \},$$ 
usually abbreviated to $H_\beta$ when the metric $m$ is obvious and/or irrelevant.
For any subset $W\subset \mathrm{T}_m \mathfrak{D}(\Pi)$, we denote by $W^+$ its positive projectivisation $\{\R_+^* w~|~ w\in W\smallsetminus \{0\}\}$ in the sphere of directions (rays) of $\mathrm{T}_m \mathfrak{D}(\Pi)$: for instance, $H_\beta^+$. 
\end{notation}

\begin{definition} \label{def:filled}
Let $\be$ be a horoball connection of a decorated hyperbolic surface $\Pi$, with universal covering map $\pi:\widetilde{\Pi} \rightarrow \Pi$.
Then the \emph{subsurface filled by} $\be$ is the quotient $\Pi_\beta$ by $\pi_1(\Pi)$ of the union of all convex hulls in $\widetilde{\Pi}$  of connected components of $\pi^{-1}(\beta)$.
This notion depends on a choice of hyperbolic metric $m$ on $\Pi$, but is independent of $m$ up to isotopy of $\Pi$.
\end{definition}
If the horoball connection $\beta$ is simple then the filled ``subsurface'' $\Pi_\beta$ is reduced to $\beta$.
Otherwise, $\Pi_\beta$ is a genuine subsurface with geodesic boundary in $\Pi$, embedded except possibly at some boundary loops of $\Pi_\beta$ that can be matched in pairs (when two of the convex hulls in $\widetilde{\Pi}$ share a boundary geodesic).
We can define similarly the subsurface filled by any finite collection of horoball connections $\beta_1, \dots, \beta_k$, via convex hulls of connected components of $\pi^{-1}(\beta_1 \cup \dots \cup \beta_k)$.

\subsubsection{Subsurfaces}
Given a hyperbolic surface $\Pi\in\pdep nr, \puncp nr, \pholed nr, \pmob nr$, selecting a subset $E$ of $n'\leq n$ spikes of~$\Pi$, including $r'\leq r$ decorated ones, defines another surface $\Pi'\in \pdep {n'}{r'}, \puncp {n'}{r'}, \pholed {n'}{r'}, \pmob {n'}{r'}$.
(Here we assume that $n'$ is large enough that $\Pi'$ is indeed a hyperbolic surface, in the sense of Section~\ref{sec:toptype}).
Restriction of hyperbolic metrics determines a smooth, natural projection $\pi:\mathfrak{D}(\Pi) \rightarrow \mathfrak{D}(\Pi')$.

If this subset $E$ of spikes contains all the decorated spikes, then given any metric $m\in \mathfrak{D}(\Pi)$ and its image $m':=\pi(m) \in \mathfrak{D}(\Pi')$, the linear projection 
$$\mathrm{d}\pi:\mathrm{T}_m\mathfrak{D}(\Pi) \rightarrow \mathrm{T}_{m'}\mathfrak{D}(\Pi')$$ restricts to a map between admissible cones: $(\mathrm{d}\pi)^{-1} (\Lambda(m'))=\Lambda(m)$. 
This is just a way of saying that admissibility of an infinitesimal deformation does not depend on the infinitesimal motions of the undecorated spikes, since these do not span any horoball connections.
For general $E$ one has only $(\mathrm{d}\pi)^{-1} (\Lambda(m')) \supset \Lambda(m)$, i.e.\ an admissible deformation of $\Pi$ projects to an admissible deformation of $\Pi'$.

If $E$ is \emph{equal} to the set of decorated spikes, i.e.\ $n'=r'=r$, then the admissible cone $\Lambda(m') \subset \mathrm{T}_{m'} \mathfrak{D}(\Pi')$ is properly convex, i.e.\ its closure contains no line: this is true because if both a vector $v$ and its negative belong to~$\overline{\Lambda(m')}$, then $\mathrm{d}l_\beta(v)=0$ for every horoball connection~$\beta$, hence $v=0$ because the lengths of horoball connections smoothly determine the metric~$m'$.
One then has $\Lambda(m) \simeq \Lambda(m') \oplus \mathbb{R}^{n-r}$ as convex sets, and likewise $\overline{\Lambda(m)} \simeq \overline{\Lambda(m')} \oplus \mathbb{R}^{n-r}$.
(In general any convex cone has such a decomposition, essentially unique, as a sum of a properly convex cone and an $\mathbb{R}$-vector space.)

\begin{remark} \label{rem:spheres}
When investigating the faces of the convex spherical polytope $\overline{\Lambda(m)}^+$ below, each face will thus be a suspension of a properly convex projective polytope with the sphere $(\mathbb{R}^{n-r})^+ \simeq \mathbb{S}^{n-r-1}$, accounting for free motion of undecorated spikes.
When $n>r$ this includes the smallest-dimensional face $\mathbb{S}^{n-r-1}$ situated ``at infinity'' (suspended with the empty set), representing deformations that \emph{only} move the nondecorated spikes.
All other faces (defined as intersections of $\overline{\Lambda(m)}^+$ with supporting great spheres) are topologically closed balls, containing $\mathbb{S}^{n-r-1}$ in their boundary.
\end{remark}
(Here $\mathbb{S}^0$ is as usual a pair of antipodal points; and we can also set $\mathbb{S}^{-1}:=\varnothing$ in order to extend patterns consistently to the fully decorated case $n=r$.)

\subsubsection{Numbers of horoball connections and closed curves} \label{sec:counting}

We can now count the simple horoball connections and simple closed curves present in the four types of surfaces:

\noindent $\bullet$ In polygons $\pdep nr$ we can index simple horoball connections by unordered pairs of distinct (decorated) spikes, forming their endpoints.  

\noindent $\bullet$ In punctured polygons $\puncp nr$ we can index simple horoball connections in \emph{ordered} pairs of spikes, possibly equal (the order is dictated by which side of the horoball connection contains the puncture).

\noindent $\bullet$ In holed polygons $\pholed nr$ we can use a similar indexation scheme, and the surface also contains a simple closed geodesic.

\noindent $\bullet$ Finally in Möbius strips $\pmob nr$, separating horoball connections are the ones disjoint from the unique simple closed (core) geodesic and can be indexed in ordered pairs of spikes (possibly equal), while nonseparating horoball connections, which cross the core, can be indexed in unordered pairs of spikes (also possibly equal). 

These tallies of simple horoball connections and simple closed curves are summarized in the rightmost column of the table in Section~\ref{sec:notation} above.

\subsection{Arc complexes}

The other main ingredient of Theorem~\ref{thm:mainA} is the arc complex of a (spiked, partially decorated) hyperbolic surface $\Pi$.
We now give the definitions.


An \emph{arc} in $\Pi$ is a proper geodesic embedding $\al$ of a closed interval or ray $I\subset \R$ into~$\Pi$, relative to its boundary (i.e.\ $\alpha(I)\cap \partial \Pi = \alpha(\partial I)$), and subject to the conditions below:
\begin{enumerate}
	\item If $I=[a,b]$, then the arc $\al$ is a finite geodesic segment connecting two sides of~$\Pi$. 
We rule out the case that these two sides are consecutive and share an undecorated spike (i.e.~that $\alpha$ separates off a single, undecorated spike from the rest of the surface).
	\item If $I=[a,\infty)$, then the arc $\alpha$ is a hyperbolic geodesic ray in $\Pi$ such that  $\al(a)\in \partial \Pi$, and the infinite end converges to a spike of $\Pi$.
    We require that that spike be decorated.
\end{enumerate}  
To emphasize these two conditions we will sometimes speak of \emph{permitted} arcs (although by default we use ``arc'' as a strict synonym).

Let $\mathscr A$ be the set of all arcs of the two types above. 
The \emph{arc complex} of the partially decorated hyperbolic surface $\Pi$ is a simplicial complex $\hac{\Pi}$ whose 0-simplices are given by the arcs in $\mathscr A$ seen up to isotopy fixing the boundary (setwise), and for $k\geq 1$, every $k$-simplex is given by a $(k+1)$-tuple of pairwise disjoint and distinct isotopy classes.
A simplex $\sigma$ is said to be \emph{filling} if the arcs corresponding to $\sigma^{(0)}$  decompose the surface into topological disks with at most one decorated spike, and punctured disks with no decorated spikes.
The \emph{pruned arc complex} of a polygon $\Pi$, denoted by $\sac \Pi$, is the union of the interiors of the filling simplices of the arc complex $\hac \Pi$.

\begin{definition} \label{def:coll}
Given a hyperbolic surface $\Pi\in\pdep nr, \puncp nr, \pholed nr, \pmob nr$,  we consider a subset $\vec{\be}=\{ \be_1, \dots, \beta_h \}$ of the collection formed by the boundary horoball connections of $\Pi$ and the closed simple loop of $\Pi$ (if present). 
The \emph{marked} arc complex, denoted by $\acp{\Pi, \vec{\be}}$, is the subcomplex of the arc complex $\acp{\Pi}$ generated by the permitted arcs that do not intersect any $\be_i$.
\end{definition}

\subsection{Strip deformations}
An arc $\alpha:I \rightarrow \Pi$ in $\Pi$ defines an infinitesimal deformation of $\Pi$, as follows.
Choose a point $p_\alpha$ in the interior of $\alpha(I)$.
Suppose, first, that $I$ is a compact interval.
We may cut $\Pi$ along the geodesic $\alpha(I)$ and insert a \emph{strip} of hyperbolic plane along $I$, of minimal width $t\geq 0$ achieved precisely at $p_\alpha$.
(A strip is the region bounded by two bi-infinite geodesics with disjoint closures.)
This moves the ideal vertices and decorations of $\Pi$, but we can use their new positions to define a new hyperbolic metric $m_t\in \mathfrak{D}(\Pi)$.

When $I$ is noncompact, say $I=[0,+\infty)$, we make a similar cut-and-paste definition of the metric $m_t$, except that the shape we insert along $\alpha(I)$ is a \emph{parabolic strip}, isometric to the region of $\HP$ bounded by two mutually asymptotic lines, with horocyclic width $t$ at the depth of $p_\alpha$.

In both cases, the construction can be performed equivariantly on all lifts of $\alpha(I)$ to the universal cover of $\Pi$: this defines also a new holonomy representation for $m_t$.
By definition, the strip deformation associated to $\alpha$ is the class of $t\mapsto m_t$ in $\mathrm{T}_m\mathfrak{D}(\Pi)$. 
Extending linearly, this lets us define, for every metric $m\in \mathfrak{D}(\Pi)$, the \emph{strip map}
\begin{equation} \label{eq:stripmap}
f:\mathcal{A}(\Pi) \longrightarrow \mathbb{P}^+(\mathrm{T}_m \mathfrak{D}(\Pi))
\end{equation}
which takes values in the weakly admissible deformations.
(Weak admissibility follows from the observation that the length of a closed curve or horoball connection can only go \emph{up} if it intersects one or more of the arcs supporting the deformation; in fact this increase can be measured trigonometrically~\cite[\S2.1]{dgk}).
Here is a key fact:
\begin{proposition} \label{prop:striphomeo}
For any $\Pi\in \pdep nr, \puncp nr, \pholed nr, \pmob nr$ endowed with a hyperbolic metric $m$, the 
strip map $f:\mathcal{A}(\Pi) \rightarrow \padm$ is a homeomorphism.
\end{proposition}
We refer to~\cite{ppballs} for a proof (the statement there is limited to $\Pi\in \pdep nr, \puncp nr$, but the proof only uses the finiteness of the arc complex $\mathcal{A}(\Pi)$). Proposition~\ref{prop:striphomeo} \emph{fails} in general for surfaces with infinite arc complexes, although $f$ still establishes a homeomorphism from the pruned arc complex $\sac \Pi$ to the interior of $\padm$.
Proposition~\ref{prop:striphomeo} contains in particular the following facts: 
\begin{itemize}
\item The largest number of vertices in a simplex $\sigma$ of the arc complex $\hac{\Pi}$ is $N_\Pi$, the dimension of the deformation space $\mathfrak{D}(\Pi)$. 
\item For any such $\sigma$, the strip deformations associated to its vertices form a basis of $\mathrm{T}_m \mathfrak{D}(\Pi)$.
\end{itemize}


\section{Main result} \label{sec:mainresult}
\begin{figure}[!h]
    \centering
    \includegraphics[width=0.9\linewidth]{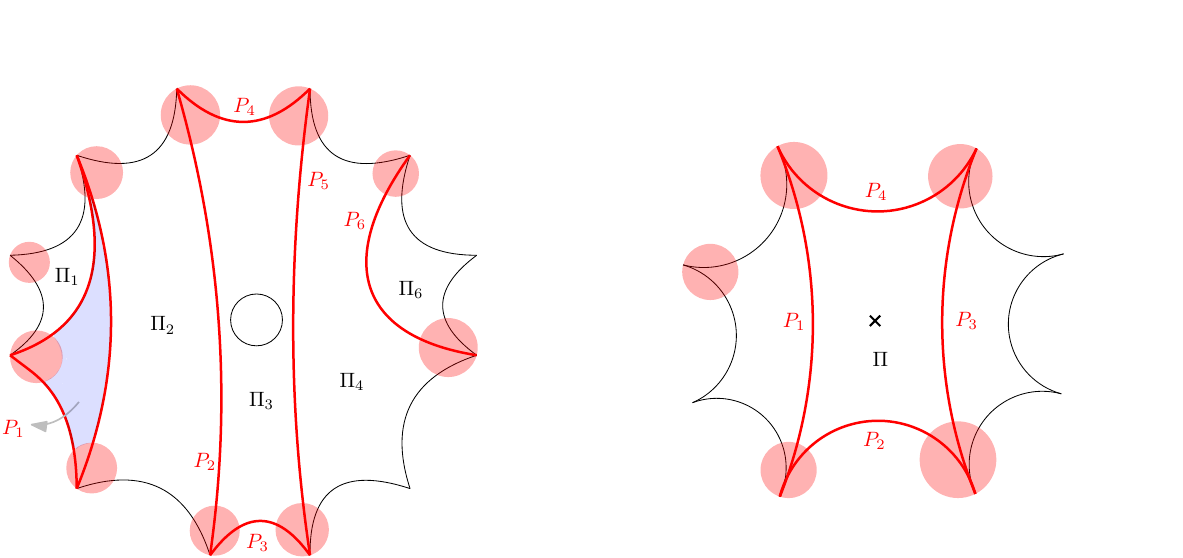}
    \caption{Left panel: A spread subset (red with blue interior) in a partially decorated crown. \\ Right panel: \emph{Not} a spread subset because $\Pi\smallsetminus P$ has a component that does not meet $\partial \Pi$.}
    \label{fig: spread}
\end{figure}

\subsection{Spread subsets}
Theorem~\ref{thm:mainB} below will express the faces of the projectivized admissible cone $\overline{\Lambda(m)}^+$ in terms of certain subcomplexes of the arc complex $\mathcal{A}(\Pi)$.
To describe the subcomplexes of interest, we need to make some definitions.
\begin{definition}\label{def: spread}
Consider $\Pi\in \pdep nr, \puncp nr, \pholed nr, \pmob nr$, and a nonempty closed subset $P\subsetneq \Pi$ with connected components $P_1, \dots, P_k$. 
We say that $P$ is \emph{spread} if 
\begin{itemize}
\item every $P_i$ is either a simple closed geodesic loop, or a horoball connection, or a subsurface with geodesic boundary and positive area whose spikes (taken among those of $\Pi$) are all decorated; and 
\item every component of $\Pi \smallsetminus P$ intersects $\partial \Pi$.
\end{itemize}
\end{definition}
\noindent
See Figure~\ref{fig: spread} for an illustration.
Given a spread subset $P\subset \Pi$, the metric completions of the connected components of $\Pi\smallsetminus P$ are hyperbolic surfaces $\Pi_1, \dots, \Pi_\kappa$.
For each $\Pi_\iota$, we let $\vec{\beta}_\iota$ be the (nonempty) collection of horoball connections and geodesic loop(s) in $\partial \Pi_\iota$ that are contained in $P$.
By construction, the simplicial complex
\begin{equation} \label{eq:suspensionsss}
\mathcal{A}(\Pi, P):=\mathcal{A}(\Pi_1, \vec{\beta}_1) \Join \dots \Join \mathcal{A}(\Pi_\kappa, \vec{\beta}_\kappa)
\end{equation}
embeds simplicially in the arc complex $\mathcal{A}(\Pi)$.
It can be described as the subcomplex of arcs that do not meet $P$.

\subsection{Statements}
Here is the main result of this paper, refining Proposition~\ref{prop:striphomeo}.

\begin{theorem} \label{thm:mainB}
Let $\Pi\in \pdep nr, \puncp nr, \pholed nr, \pmob nr$ be a hyperbolic surface with $n$ spikes and $r\leq n$ decorations, with metric $m$. Then: 
\begin{itemize} 
\item The positively projectivized, weakly admissible deformations in $\mathbb{P}^+(\mathrm{T}_m\mathfrak{D}(\Pi))$ form a finite compact polytope $\padm$, homeomorphic to the arc complex $\mathcal{A}(\Pi)$ via the strip map $f$ of~\eqref{eq:stripmap}.
\item The faces of $\padm$ are exactly the $f(\mathcal{A}(\Pi, P))$, where $P \subset \Pi$ ranges over spread subsets. 
\item For any spread subsets $P,P' \subsetneq \Pi$,
\begin{equation} \label{eq:inkface}
\left . \begin{array}{lll} \mathrm{(i)} & \mathcal{A}(\Pi, P) \subset \mathcal{A}(\Pi, P') & \iff P \supset P' \\ \mathrm{(ii)} & \mathcal{A}(\Pi, P) \subsetneq \mathcal{A}(\Pi, P') & \iff P \supsetneq P'. \end{array} \right \}
\end{equation}
\end{itemize}
\end{theorem}



The dimension and topology of a face $f(\mathcal{A}(\Pi, P))$ can be deduced from homeomorphicity of $f$ and from the following, strictly topological result.
\begin{proposition} \label{prop:dimofmac}
Let $\Pi\in\pdep nr, \puncp nr, \pholed nr, \pmob nr$ be a partially decorated surface, and $\vec{\be}=\{ \be_1, \dots, \beta_h \}$ a collection as in Definition~\ref{def:coll} (hence $h \leq n+1$). 
Let $N:=N_\Pi-1-h$ (see Section~\ref{sec:notation}).
Then:
\begin{enumerate}
\item For $\Pi\in \pdep nr$,
the marked arc complex $\acp{\Pi, \vec{\be}}$ is PL-homeomorphic to:
\begin{enumerate}
\item $\varnothing$ if $(r,h)=(n,n)$;
\item the sphere $\mathbb{S}^{N}$ if $(r,h)=(0,0)$ or $(1,0)$ or $(2,1)$;
\item the closed ball ${\mathbb{B}}^{N}$ otherwise.
\end{enumerate}
\item For $\Pi\in \puncp nr$,
the marked arc complex $\acp{\Pi, \vec{\be}}$ is PL-homeomorphic to:
\begin{enumerate}
\item $\varnothing$ if $(r,h)=(n,n)$;
\item the sphere $\mathbb{S}^{N}=\mathbb{S}^{n-2}$ if $(r,h)=(0,0)$;
\item the closed ball ${\mathbb{B}}^{N}$ otherwise.
\end{enumerate}
\item For $\Pi\in \pholed nr$,
the marked arc complex $\acp{\Pi, \vec{\be}}$ is PL-homeomorphic to: \begin{enumerate}
\item $\varnothing$ if $(r,h)=(n,n+1)$;
\item the sphere $\mathbb{S}^{N}=\mathbb{S}^{n-2}$ if $(r,h)=(0,1)$;
\item the closed ball $\mathbb{B}^{N}$ otherwise.
\end{enumerate}
\item For $\Pi\in \pmob nr$, 
the marked arc complex $\acp{\Pi, \vec{\be}}$ is PL-homeomorphic to: \begin{enumerate}
\item $\varnothing$ if $(r,h)=(n,n+1)$ or $(n,n)$ and $\vec{\beta}$ contains all boundary horoball connections;
\item the sphere $\mathbb{S}^{N}=\mathbb{S}^{n-2}$ if $(r,h)=(0,1)$;
\item the closed ball ${\mathbb{B}}^{N}$ otherwise.
\end{enumerate}
\end{enumerate}
\end{proposition}
\begin{remark}
Spherical cases: in case 1.(b), when $(r,h)=(2,1)$ the $r=2$ decorated spikes are necessarily consecutive, connected by the single horoball connection $\beta_1=\beta_h \in \vec{\beta}$. 
In cases 3.(b) and 4.(b), the collection $\vec{\beta}$ is necessarily reduced to the core curve. 
\end{remark}
\begin{remark} \label{rem:spheres2}
To determine $\mathcal{A}(\Pi, P)$ for a spread subset $P\subset \Pi$ as in Theorem~\ref{thm:mainB}, one needs to apply Proposition~\ref{prop:dimofmac} to the components $(\Pi_\iota, \vec{\beta}_\iota)$ of $\Pi\smallsetminus P$, and combine results by suspension as in the formula~\eqref{eq:suspensionsss}. 
In such applications, the ``empty'' cases (a) of the Proposition never arise, since by definition of spread subsets the component of $\Pi\smallsetminus P$ corresponding to $\Pi_\iota$ intersects $\partial \Pi$ (along an edge that cannot belong to $P$) so $\vec{\beta}_\iota$ cannot be all of $\partial \Pi_\iota$. The ``spherical'' cases (b) occur only in very specific situations: 
\begin{itemize}
\item 1.(b) when $\Pi_\iota$ is cut off from $\Pi$ by a single boundary-parallel horoball connection that skips only undecorated horoballs (hence $(r_\iota, h_\iota)=(2,1)$).
This happens for \emph{every} component $\Pi_\iota \subset \Pi-P$ if and only if $r>0$ and $P$ is largest possible, namely equal to the convex hull of all decorated punctures of $\Pi$. 
Then, the corresponding sphere $f(\mathcal{A}(\Pi, P)) \simeq \mathbb{S}^{n-r-1}$ represents all possible motions of the nondecorated spikes. 
See also Remark~\ref{rem:spheres}. 
For any other $P$ the definition of $\mathcal{A}(\Pi, P)$ as iterated suspension~\eqref{eq:suspensionsss} contains at least one ball term, and is therefore a ball.
\item 3.(b) and 4.(b) when $\Pi$ is undecorated ($r=0$) and $P$ is reduced to the core loop $c$.
Then, $\overline{\Lambda(m)}^+$ is a whole hemisphere $\simeq {\mathbb{B}}^{n-1}$, defined by the lengthening of $c$, with boundary the sphere $f(\mathcal{A}(\Pi, P)) \simeq \mathbb{S}^{n-2}$; note that $\mathcal{A}(\Pi, P)$ is then simplicially equivalent to $\mathcal{A}(\puncp n0)$.
\end{itemize} 
Note that our definition of spread sets $P \subsetneq \Pi$ supposes $P\neq \varnothing$.
If we did allow $P=\varnothing$, then $\mathcal{A}(\Pi, \varnothing)$ would be simply the whole arc complex $\mathcal{A}(\Pi)$, which is a sphere exactly when $\Pi \in \pdep n0, \pdep n1, \puncp n0$ (polygons decorated at most once and undecorated punctured polygons), as reflected also in Cases 1.(b) and 2.(b). The admissible cone is then all of $\mathrm{T}_m(\mathfrak{D}(\Pi))$, because there are no closed geodesics or horoball connections.
\end{remark}

\begin{example} To illustrate Theorem~\ref{thm:mainB} and Proposition~\ref{prop:dimofmac}, in Appendix~\ref{sec:app2} we show a situation (the fully decorated 2-crown) where the projectivized admissible cone is a polyhedron of dimension 3, and list all spread subsets and correponding subcomplexes associated to its faces, edges and vertices. 
\end{example}

\subsection{Proofs of Theorem~\ref{thm:mainB} and Proposition~\ref{prop:dimofmac}} \label{sec:mainproofs}

\begin{proof}[Proof of Theorem~\ref{thm:mainB} assuming Proposition~\ref{prop:dimofmac}]
The first statement of Theorem~\ref{thm:mainB} is just a paraphrase of Proposition~\ref{prop:striphomeo}. 
To prove the second and third statement, consider a spread subset $P \subsetneq \Pi$, with connected components $P_1, \dots, P_k$ and complementary components $\Pi_1, \dots, \Pi_\kappa$.
(Although the proof below covers all cases $\Pi \in \pdep nr, \puncp nr, \pholed nr, \pmob nr$, it may be helpful to first focus on the paradigmatic case of polygons, $\Pi \in \pdep nr$, where the $P_i$ are horoball connections and polygons, and the $\Pi_\iota$ are polygons.)

Any infinitesimal deformation $v$ of the metric $m\in \mathfrak{D}(\Pi)$ induces, by restrictions, infinitesimal deformations $v_i$ and $v_\iota$ of the induced metrics $m_i\in \mathfrak{D}(P_i)$ and $m_\iota \in \mathfrak{D}(\Pi_\iota)$.
(If $P_i$ is a simple closed loop or horoball connection then $v_i$ just records its infinitesimal length variation.) 
These deformations $v_i, v_\iota$ are \emph{compatible}, in the following sense: for every component $\Pi_\iota$ and every $\beta \in \vec{\beta}_\iota$, one has
$$\mathrm{d} l_\beta (m_\iota)(v_\iota) = \mathrm{d} l_\beta (m_i)(v_i)$$ where $i$ is the index of the component $P_i$ of $P$ that contains the boundary component $\beta$ of $\Pi_\iota$. 

Conversely, any system of infinitesimal deformations $v_i\in \mathrm{T}_{m_i} \mathfrak{D}(P_i)$ and $v_\iota \in \mathrm{T}_{m_\iota} \mathfrak{D}(\Pi_\iota)$ defines an infinitesimal deformation $v\in \mathrm{T}_m \mathfrak{D}(\Pi)$ if and only if the $v_i$ and $v_\iota$ are compatible.
This vector $v\in \mathrm{T}_m \mathfrak{D}(\Pi)$ is then~\emph{unique}, because there is only one isometric identification between the boundaries of the $P_i$ and $\Pi_\iota$ that matches the decoration horoballs together. 

(\underline{Remark}: when $P_i$ is a simple closed curve, there is a $1$-parameter family of ways to glue it isometrically to the adjacent component $\Pi_\iota$; however these gluings all produce the \emph{same} metric $m$ because there is \emph{only one} adjacent component $\Pi_\iota$. 
This is a special feature of the surfaces $\pdep nr, \puncp nr, \pholed nr, \pmob nr$, because they do not contain any nonperipheral two-sided simple closed curve $\beta$, i.e.\ no ``twist''. 
It would fail in more general surfaces.)

\smallskip

Compatibility imposes $|\vec{\beta}_\iota|$ linearly independent constraints on the infinitesimal deformation $v_\iota$ of $\Pi_\iota$.
The subspace $V_P \subset \mathrm{T}_m \mathfrak{D}(\Pi)$ of infinitesimal deformations that keep $P$ rigid, i.e.\ induce $v_i=0\in \mathrm{T}_{m_i} \mathfrak{D}(P_i)$ for each $P_i$, therefore satisfies 
\begin{align*} \dim (V_P) &=  0 + \textstyle \sum_{\iota=1}^\kappa N_{\Pi_\iota} - |\vec{\beta}_\iota| & \text{by the above compatibility condition} \\
& = \textstyle \sum_{\iota=1}^\kappa 1+\dim \mathcal{A}(\Pi_\iota, \vec{\beta}_\iota) & \text{by Proposition~\ref{prop:dimofmac} applied to $\Pi_\iota$} \\
& = 1+ \dim \mathcal{A}(\Pi, P) & \text{by \eqref{eq:suspensionsss}.}
\end{align*}
(In the second step we used the observation, from Remark~\ref{rem:spheres2}, that the ``empty'' cases~(a) of Proposition~\ref{prop:dimofmac} never arise for the $\Pi_\iota$.)

Recall the strip map $f: \mathcal{A}(\Pi) \rightarrow \mathbb{P}^+(\mathrm{T}_m \mathfrak{D}(\Pi))$ of~\eqref{eq:stripmap}. 
Clearly $$f(\mathcal{A}(\Pi, P)) ~ \subset ~ V_P^+ \cap \overline{\Lambda(m)}^+,$$ because the arcs in $\mathcal{A}(\Pi, P)$ do not meet $P$. 
The above dimension computation shows that $f(\mathcal{A}(\Pi, P))$ has full dimension in $V_P^+$.

Let us now prove the second point of Theorem~\ref{thm:mainB}.
Consider a boundary point $y\in \partial \overline{\Lambda(m)}^+$ of the positively projectivized cone of weakly admissible deformations.
By Proposition~\ref{prop:striphomeo}, we can write $y=f(x)$ for some (unique) boundary point $x\in \mathcal{A}(\Pi) \smallsetminus \mathcal{PA}(\Pi)$ of the projectivized arc complex.
By definition, the support $|x|$ of $x$ is a collection of disjoint arcs that does not fill~$\Pi$.
Let $\mathcal{H}_x$ be the (nonempty) collection of simple horoball connections and simple closed curves disjoint from~$|x|$.
The face $F_x$ of $\overline{\Lambda(m)}^+$ containing $y=f(x)$ is the intersection of all the (positively projectivized) hyperplanes $H_\beta^+$ for $\beta$ ranging over $\mathcal{H}_x$.

Let $P_x \neq \varnothing$ be the subsurface filled by the union of all elements of $\mathcal{H}_x$.
We claim that $P_x$ is a spread subset disjoint from the support $|x|$.
The connected components of $P_x$ are clearly simple closed curves and horoball connections and fully decorated spiked subsurfaces, by Definition~\ref{def:filled} and the discussion that follows it. 
The definition also makes it clear that $P_x \cap |x|=\varnothing$, because if an arc of $|x|$ crosses a component of $P_x$ then it must also cross one of the horoball connections and curves that this component is the convex hull of.
So we only have to show that every component $\Pi_\iota$ of $\Pi \smallsetminus P_x$ intersects $\partial \Pi$.
If (by contradiction) that were not the case, then every boundary component (horoball connection or boundary loop) of $\Pi_\iota$ would be in $P_x$. 
Therefore no arc of $|x|$ could cross $\Pi_\iota$, lest it intersect one of those boundary components and hence also $P_x$ (contradicting the definition of $P_x$). 
At least one of the boundary components of $\Pi_\iota$ must be a horoball connection, hence we can issue out of one of its endpoints another horoball connection $\beta$ in $\Pi_\iota \sqcup P_x$, intersecting (possibly contained in) $\Pi_\iota$ and disjoint from $|x|$.
Such a $\beta$ should definitionally belong to $\mathcal{H}_x$, hence be contained in $P_x$: contradiction. 
Therefore, $P_x$ is a spread subset disjoint from $|x|$.

By the dimension computation above, we now have that $f(\mathcal{A}(\Pi, P_x))$ contains an open subset of the face $F_x \ni f(x)$; in fact the restriction $f|_{\mathcal{A}(\Pi, P_x)} : \mathcal{A}(\Pi, P_x) \rightarrow F_x$ is an open map. 
For every point $x'\in F_x$ we have $P_{x'}=P_x$, because the supports $|x'|$ and $|x|$ cross the same curves and horoball connections, by definition of $F_x$. 
Therefore $f|_{\mathcal{A}(\Pi, P_x)} : \mathcal{A}(\Pi, P_x) \rightarrow F_x$ is a homeomorphism, proving the second point of Theorem~\ref{thm:mainB}.

Finally, let us now prove the last point~\eqref{eq:inkface} of Theorem~\ref{thm:mainB}.
Point~\eqref{eq:inkface}.(ii) follows from~\eqref{eq:inkface}.(i), since $X\subsetneq Y$ is equivalent to $X\subset Y$ and $X\not\supset Y$ (for any sets $X,Y$). 
The ``$\Leftarrow$'' direction of~\eqref{eq:inkface}.(i) is clear, since $\mathcal{A}(\Pi, P)$ is the subcomplex of arcs that do not meet $P$.
For the ``$\Rightarrow$'' direction, let us prove the contrapositive: suppose $P \not \supset P'$, i.e. there exists $x\in P'_i \cap \Pi_\iota$ (for some components $P'_i$ of $P'$ and $\Pi_\iota$ of $\Pi\smallsetminus P$); and let us find an element of $\mathcal{A}(\Pi, P) \smallsetminus \mathcal{A}(\Pi, P')$. 
It suffices to find a permitted arc $\alpha$ in $\Pi_\iota$ that crosses $P'_i$: this can be done by a routine discussion of cases, using the fact that $\Pi_\iota$ intersects $\partial \Pi$ (i.e.\ $P$ is spread) to allow $\alpha$ to exit the surface.
\end{proof}

\begin{proof}[Proof of Proposition~\ref{prop:dimofmac}]
We start with Case 1.
Given a partially decorated polygon $\Pi\in \pdep nr$ and a collection $\vec{\beta}$ of $h$ sides of $\Pi$ that are horoball connections, we first form a disk $\Pi'$ with $r+(n-h)$ points on its boundary, colored red ($r$ points) and blue ($n-h$ points).
Red corresponds to decorated spikes of $\Pi$, while Blue corresponds to sides of $\Pi$ that do \emph{not} belong to $\vec{\beta}$ (regardless of whether their endpoints are decorated).
There is a unique such $\Pi'$ respecting the cyclic order, and it is not hard to check that $\Pi'$ conversely determines $\Pi$ (a pair of consecutive red points corresponds to a side in $\vec{\beta}$, while a pair of consecutive blue points corresponds to an undecorated spike).

The disk $\Pi'$ has $(r+n-h)(r+n-h-3)/2$ diagonals, which form the vertices of a simplicial sphere $S\simeq \mathbb{S}^{r+n-h-4}$, called the \emph{diagonal complex} (dual to the associahedron), when we let each $(k+1)$-tuple of disjoint diagonals span a $k$-simplex. 
Let $S'$ be the subcomplex of $S$ spanned only by blue-blue and red-blue diagonals.
We have the following lemma, taken from~\cite{ppballs}.
\begin{lemma} \label{lem:blured}
We have $S'=\varnothing$ if all vertices are red; $S'=S$ if there is no red-red diagonal; and $S' \simeq {\mathbb{B}}^{r+n-h-4}$ in all other cases.
\end{lemma}
Case 1 of Proposition~\ref{prop:dimofmac} follows immediately from the lemma and from the definition of the coloring, because $S'$ is by construction simplicially equivalent to $\mathcal{A}(\Pi, \vec{\beta})$. 
(Note that when $(r,h)=(2,1)$ there are two red vertices but they are consecutive, hence do not span a diagonal.)

\smallskip

The proof of Case 2 is similar: given a surface $\Pi\in \puncp nr$ and the collection $\vec{\beta}$ of $h$ boundary horoball connections of $\Pi$, we build a punctured disk $\Pi'$ with $r$ red and $n-h$ blue vertices and define its diagonal complex $S\simeq \mathbb{S}^{r+n-h-2}$, dual to the cyclohedron, with $(r+n-h)(r+n-h-1)$ vertices (loops around the puncture connecting a vertex to itself do count as ``diagonals'', but edges of $\Pi'$ do not). We define similarly $S'$ as the subcomplex spanned by blue-blue and red-blue diagonals. 
The analogue of Lemma~\ref{lem:blured} holds, with only ${\mathbb{B}}^{r+n-h-4}$ changed to ${\mathbb{B}}^{r+n-h-2}$: see~\cite{ppballs}.
Case 2 follows.

\smallskip

For Case 3, we distinguish two subcases. First, if the boundary loop $c$ belongs to $\vec{\beta}$, then 
$\mathcal{A}(\pholed nr, \vec{\beta})$ identifies with 
$\mathcal{A}(\puncp nr, \vec{\beta}\smallsetminus \{c\})$ so we are reduced to Case 2; this covers in particular the implications of 3.(a) and 3.(b). 
If $c \notin \vec{\beta}$, we build from $\Pi \in \pholed nr$ a disk $\Pi'$ with $r$ red and $n-h$ blue vertices on its boundary, and an extra  blue vertex in the center.
The diagonal complex $S$ is now formed by boundary-to-boundary \emph{and} boundary-to-center arcs, $(r+n-h)^2$ arcs in total. 
It is shown in~\cite{ppstrong} (see also Theorem~\ref{strcollhole} below) that the subcomplex $S'\subset S$ spanned by blue-blue and blue-red arcs is always homeomorphic to ${\mathbb{B}}^{r+n-h-1}$ (in particular, it always contains the top-dimensional simplex spanned by all center-to-boundary arcs).
Case 3 follows.

\smallskip

For Case 4, similarly if the core curve $c$ belongs to $\vec{\beta}$, then $\mathcal{A}(\pmob nr, \vec{\beta})$ identifies with $\mathcal{A}(\puncp nr, \vec{\beta}\smallsetminus \{c\})$ because the complement of the core curve of the M\"obius strip is an annulus; and therefore we are reduced to Case 2: this covers in particular the implication in Case 4.(b).
If $c \notin \vec{\beta}$, we consider a M\"obius strip $\Pi'$ with $r$ red and $n-h$ blue vertices on its boundary, and the associated diagonal complex $S$ (containing $(r+n-h)(r+n-h-1)$ separating diagonals, which are isotopic into the boundary without being boundary edges; and $(r+n-h)(r+n-h+1)/2$ nonseparating ones, which cross the core curve). 
It is shown in~\cite{ppstrong} (see also Theorem~\ref{thm: strcoll} below) that unless all vertices are red, the subcomplex $S'\subset S$ spanned by blue-blue and red-blue  diagonals is homeomorphic to~${\mathbb{B}}^{r+n-h-1}$.
Case 4 follows, and Proposition~\ref{prop:dimofmac} is proved.

Note that the triangulations of closed balls given by the above references are all \emph{collapsible}, a strong combinatorial condition. 
We do not use that fact in this paper.
\end{proof}

\section{The facets of $\padm$} \label{sec:walls}
\noindent
Consider a hyperbolic surface $\Pi \in \pdep nr, \puncp nr, \pholed nr, \pmob nr$, a hyperbolic metric $m$ on $\Pi$, and the polyhedron $\overline{\Lambda(m)}^+$ of weakly admissible deformations.
By definition, a \emph{facet} of $\overline{\Lambda(m)}^+$ is a face of codimension $1$.
In this section we study which horoball connections (and geodesic loops) define facets: in other words, we classify the situations in which \eqref{eq:suspensionsss} gives a complex of dimension $N_\Pi-2$.
\begin{figure}[h!]
\centering
\includegraphics[width=0.8\linewidth]{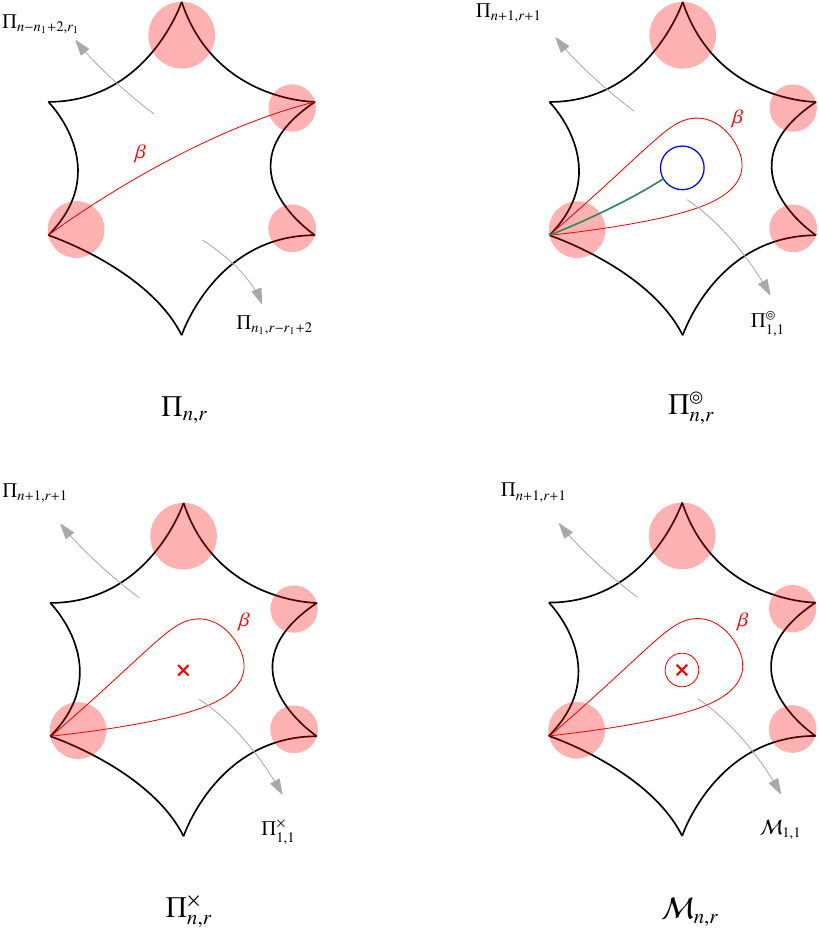}
\caption{Decomposition of the four surfaces by a horoball connection $\be$.}
\label{fig: decompose}
\end{figure}

\begin{lemma}\label{lem: simple}
The hyperplane $H_\be ^+$ intersects $\padm$ along a facet for every simple horoball connection $\be$ except when $\Pi=\pmob nr$ and the horoball connection $\be$ is maximal (Definition~\ref{def:maximalHC}).
 \end{lemma}
\begin{remark}
There are exactly two types of simple horoball connections that are \emph{not} spread subsets: 

\noindent $\bullet$~ When $\beta$ is maximal in a Möbius strip $\Pi=\pmob nr$, hence bounds off a 1-spiked Möbius strip $\pmob 11$ disjoint from $\partial \Pi$ (i.e. the exceptions in Lemma~\ref{lem: simple});

\noindent $\bullet$~ When $\beta$ is maximal in punctured polygon $\Pi=\puncp nr$, hence bounds off a punctured monogon~$\puncp 11$. Such a $\beta$ \emph{does} define a facet of the polyhedron $\overline{\Lambda(m)}^+$ (associated to the spread subset $P:=\puncp 11$), essentially because $\puncp 11$ is rigid; but see Lemma~\ref{lem:nonsimplefacet} below for a different way in which such $\beta, P$ are exceptional.
\end{remark}
 \begin{proof}[Proof of Lemma~\ref{lem: simple}]
 Let $\beta$ be a simple horoball connection.
 First, we suppose that $\beta$ is separating.
 Then $\beta$ decomposes the surface into two subsurfaces which we denote by $\Pi^1$ and $\Pi^2$. 
  The projective hyperplane $H_\be ^+$, of dimension $N-2$, intersects $\padm$ along a facet if and only if
  \[\dim \left ( H_\be ^+\cap \padm \right ) = N-2 .\]
The left hand side is at least $\dim \left (\acp{\Pi^1,  \be} \Join \acp{\Pi^2, \be} \right )$: indeed, 
any point in $\acp{\Pi^1,  \be} \Join \acp{\Pi^2, \be}$ gives rise to a strip deformation whose support is disjoint from $\beta$, hence which belongs to $H_\be ^+\cap \padm$; and the component strip deformations are linearly independent by \cite{ppstrip}.

Thus, let us compute the dimensions of $\acp{\Pi^1,  \be}$ and $\acp{\Pi^2, \be}$, using Proposition~\ref{prop:dimofmac}.
We treat the four types of surfaces in turn.
\begin{itemize}
\item $\Pi\in \pdep nr$: In this case, $\dim H_\be ^+=N-2=n+r-5$. 
There exist integers $n_1$ and $r_1$ such that $\Pi^1\in \pdep{ n_1}{r-r_1+2}$ and $\Pi^2\in \pdep{n- n_1+2}{r_1}$.
Then we have that $$
\begin{array}{l}
	\acp{\Pi^1, \be} \simeq \ball{n_1+r-r_1+2-4-1} = \ball {n_1+r-r_1-3},\\
	\acp{\Pi^2, \be} \simeq \ball{n-n_1+r_1-3}.
\end{array}
$$
Here, we have that $3\leq n_1 \leq n-1$.
Since $\be$ connects two distinct decorated spikes, we also have that $2\leq r_1\leq r$.
So for $i=1,2$, we get that $\dim \acp{\Pi^i,  \be}\geq 0.$
Hence,
\[  
\begin{array}{l}
\dim \left ( \acp{\pdep{n- n_1+2}{r_1} , \be} \Join \acp{\pdep{n_1}{r-r_1+2}, \be} \right )\\
=(n-n_1+r_1-3)+(n_1+r-r_1-3)+1\\
=n+r-5.
\end{array}
  \] 
Therefore, $ H_\be ^+$ intersects $\padm$ along a facet. 
This actually finishes the case $\Pi\in \pdep nr$, as all simple horoball conections are separating in such $\Pi$.

\item $\Pi\in \puncp nr$: In this case, $\dim H_\be ^+=n+r-3$. 
There exist integers $n_1$ and $r_1$ such that one of the subsurfaces is an ideal polygon $\Pi^1 \in \pdep{n- n_1+2}{r_1}$ and the other one is a once-punctured polygon $\Pi^2\in\puncp{ n_1}{r-r_1+2}$. Again, we have that $$
\begin{array}{l}
\acp{\Pi^1,  \be} \simeq \ball {n-n_1+r_1-3},\\
\acp{\Pi^2,  \be} \simeq \ball{n_1+r-r_1+2-2-1}=\ball{n_1+r-r_1-1}.
\end{array}
$$ 
Here, we have $1\leq r_1\leq r+1$ and $1\leq n_1 \leq n-1$.
 When $n_1=1$ and $r_1=r+1$, we have that $\acp{\Pi^2,  \be}$ is empty.
This is illustrated in the right panel of the first row of Fig.~\ref{fig: decompose}.
In that case, we get $\dim H_\be ^+ \cap \padm=n-1+r+1-3=n+r-3$. 

When $n_1\geq 2$ and $r_1\leq r$, we have
\[  
\begin{array}{l}
	\dim \left ( \acp{\pdep{n- n_1+2}{r_1},  \be} \Join \acp{\puncp{ n_1}{r-r_1+2},  \be} \right )\\
	=(n-n_1+r_1-3)+(n_1+r-r_1-1)+1\\
	=n+r-3.
\end{array}
\] 
Hence we have that $ H_\be ^+$ intersects $\padm$ along a facet.
Again this finishes the case $\Pi\in \puncp nr$, as all simple horoball connections are separating for such $\Pi$.

\item $\Pi\in \pholed nr$:  In this case, $\dim H_\be ^+=n+r-2$.
There exist integers $n_1, r_1$ such that one of the subsurfaces is an ideal polygon $\Pi^1 \in \pdep{n- n_1+2}{r_1}$ and the other one is a one-holed polygon $\Pi^2\in\pholed{ n_1}{r-r_1+2}$.
Again, we have that $$
\begin{array}{l}
	\acp{\Pi^1,  \be} \simeq \ball {n-n_1+r_1-3},\\
	\acp{\Pi^2,  \be} \simeq \ball{n_1+r-r_1+2-1-1}=\ball{n_1+r-r_1}.
\end{array}
$$ 
Here, we have $1\leq r_1\leq r+1$ and $1\leq n_1 \leq n-1$.
 This implies that the two marked arc complexes are always non-empty.
The case $n_1=1, r_1=r+1$ is illustrated in the left panel of the second row of Fig.~\ref{fig: decompose}.
In that case, the marked arc complex $\acp{\Pi^2,  \be}$ is just one point given by the isotopy class of the green arc in the figure. 

Thus, we have
\[  
\begin{array}{l}
	\dim \left (\acp{\pdep{n- n_1+2}{r_1}, \be} \Join \acp{\puncp{ n_1}{r-r_1+2},  \be} \right )\\
	=(n-n_1+r_1-3)+(n_1+r-r_1)+1\\
	=n+r-2.
\end{array}
\] 
Hence we have that $ H_\be ^+$ intersects $\padm$ along a facet.
Again this finishes the case $\Pi\in \pholed nr$, as all simple horoball connections are separating in such $\Pi$.
\begin{figure}[h!]
	\centering
\includegraphics[width=0.9
\linewidth]{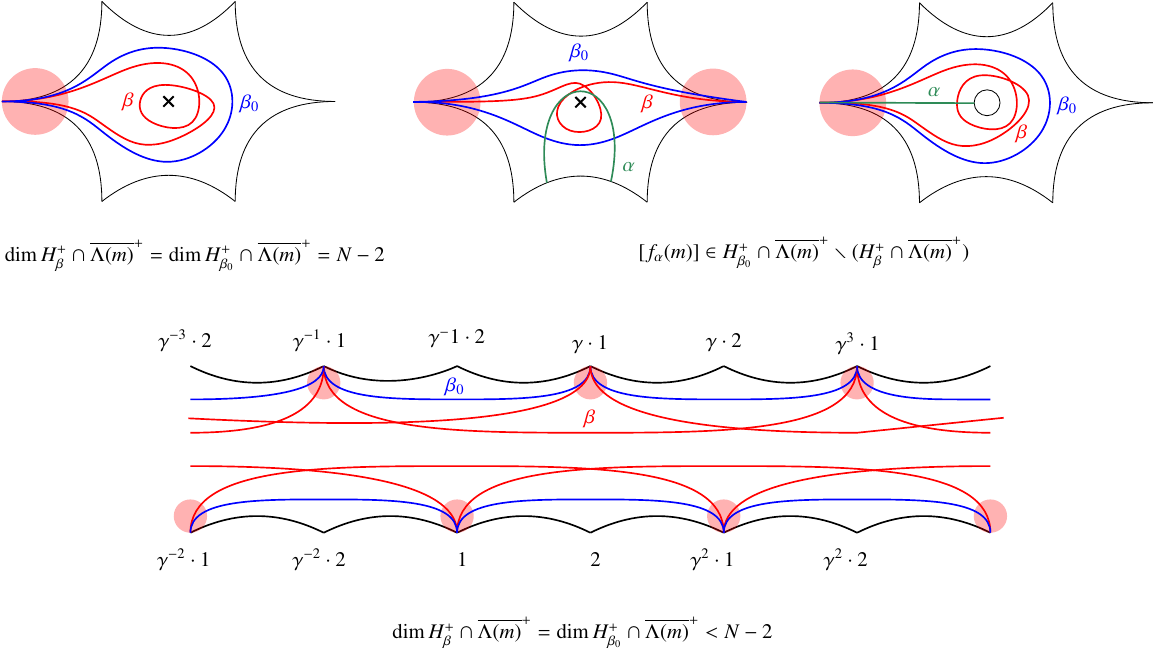}
	\caption{Lemma \ref{lem:nonsimplefacet} -- Top left: The only type of case when $H_\be^+$, for a non-simple horoball connection, forms a facet of $\padm$. Top middle and right: The hyperplane $H_\be^+$ intersects $\padm$ along a non-facet.
Bottom panel: Here the universal cover of $\pmob 21$ is shown, along with the lifts of $\be, \be_0$. Any permitted arc disjoint with $\be$ is disjoint with $\be_0$ and viceversa. But the corresponding hyperplanes $H_\beta^+,H_{\beta_0}^+$ do not form facets. }

\label{fig: fillingsurf}
\end{figure}
\item $\Pi\in \pmob nr$:  In this case, $\dim H_\be ^+=n+r-2$.
There exist integers $n_1, r_1$ such that one of the subsurfaces is an ideal polygon $\Pi^1 \in \pdep{n- n_1+2}{r_1}$ and the the other is a M\"obius strip $\Pi^2\in\pmob{n_1}{r-r_1+2}$.
Again, we have that $$
\begin{array}{l}
	\acp{\Pi^1,  \be} \simeq \ball {n-n_1+r_1-3},\\
	\acp{\Pi^2,  \be} \simeq \ball{n_1+r-r_1+2-1-1}=\ball{n_1+r-r_1}.
\end{array}
$$ 
Here, we have $1\leq r_1\leq r+1$ and $1\leq n_1 \leq n-1$.
When $n_1=1$ and $r_1=r+1$, we have that the marked arc complex $\acp{\Pi^2, \be}$ is empty (illustrated in the right panel of the second row of Fig.~\ref{fig: decompose}).
In that case, we get $\dim H_\be ^+ \cap \padm=n+r-3<\dim H_\be ^+$.
So in this case, the hyperplane $H_\be ^+$ does not meet the boundary of the admissible cone 
along a facet.
When $n_1\geq 1,$ we have that
\[  
\begin{array}{l}
	\dim \left ( \acp{\pdep{n- n_1+2}{r_1},  \be} \Join \acp{\puncp{ n_1}{r-r_1+2}, \be } \right )\\
	=(n-n_1+r_1-3)+(n_1+r-r_1)+1\\
	=n+r-3.
\end{array}
\] 
Hence we have that $ H_\be ^+$ intersects $\padm$ along a facet for every separating horoball connection, except when $\be$ is maximal.
\end{itemize}
\begin{figure}[h!]
	\centering
	\includegraphics[width=0.9\linewidth]{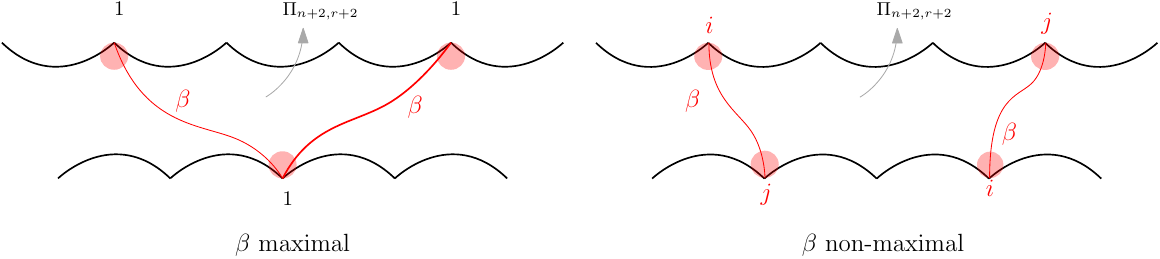}
	\caption{Decomposition of a M\"obius strip by a non-separating horoball connection $\be$.}
	\label{fig: decompose2}
\end{figure}

It remains to treat the case that $\beta$ is non-separating. 
We can then assume that $\Pi\in \pmob nr$ (Möbius strip), and that $\be$ is a horoball connection joining two decorated spikes $1\leq i\leq j\leq n$ across the core loop, and $\dim H_\be ^+=N-2=n+r-2$. See Fig.~\ref{fig: decompose2} for the two cases: $i=j=1$ and $i<j$.
In either case, the complement is a decorated ideal polygon $\pdep_{n+2,r+2}$ with two marked horoball connections.
Then, we have that $\dim \hyp\cap\padm= \dim (\pdep{n+2}{r+2}, \vec{\be})=n+r-2$, where $\vec{\be}$ consists of two copies of $\be$.
Hence $\hyp$ forms a facet of $\padm$ for every non-separating horoball connection.
This concludes the proof.
 \end{proof}

\begin{lemma}
For two distinct simple horoball connections $\be, \be'$ that define facets of $\padm$, these facets are themselves distinct.
\end{lemma}
\begin{proof}
This follows e.g.\ from the second and third points of Theorem~\ref{thm:mainB}, or from the fact that we can find a permitted arc that crosses $\beta$ but does not cross $\beta'$.
\end{proof}

\begin{lemma} \label{lem:nonsimplefacet}
	Suppose $\Pi \in \pdep nr, \puncp nr, \pholed nr, \pmob nr$.
    Given a non-simple horoball connection $\be$ of the surface $\Pi$, the hyperplane $\hyp$ does not form a facet of $\padm$, unless $\Pi\in \puncp nr$ and $\be$ is a horoball connection with identical endpoints that turns $k\geq2 $ times around the puncture.
\end{lemma}
\begin{proof}
	We need to prove that for every non-simple horoball connection $\be$, we have 
	\[ \dim \hyp \cap \padm = \dim \hyp,  \] if and only if $\beta$ is a horoball connection of $\puncp nr$ with equal endpoints that turns around the puncture.

    Firstly, we remark that non-simple horoball connections exist only for $\Pi\in \{\puncp nr, \pholed nr, \pmob nr\}$. These are the horoball connections that turn $k$ times around the puncture or core curve. 
	Let $S$ be the subsurface filled by the non-simple horoball connection $\be$.
If $S$ is the entire surface $\Pi$ then $\hyp \cap \adm^+ =\varnothing$, in which case the lemma holds trivially.
So we may assume that $S$ is strictly contained in $\Pi$. 
Then $S$ is one of the four types of surfaces. See Fig.~\ref{fig: fillingsurf}.
Let $\be_0$ be a boundary horoball connection of $S$.
Then $H_\beta^+ \cap \padm \subset H_{\be_0} \cap \padm$, because if the support of a strip deformation is disjoint from $\beta$ it is disjoint from $\beta_0$.
Therefore \[ \dim \hyp \cap \padm \leq \dim H_{\be_0}^+ \cap \padm .\] The equality holds only when any permitted arc disjoint with $\be_0$ is disjoint with $\be$. In Fig.~\ref{fig: fillingsurf}, this is illustrated in the top left panel and in the bottom panel.  in either of the two cases:
	\begin{itemize}
\item When $S=\puncp 11$ and $\be$ is a horoball connection that turns $k\geq1 $ times around the puncture (with equal endpoints).
\item When $S=\pmob 11$ and $\be$ is a horoball connection that turns $k\geq1$ times around the core curve (with equal endpoints).
	\end{itemize}
	In the second case, from the proof of Lemma \ref{lem: simple} we know that $\dim H_{\be_0}^+ \cap \padm < \dim H_{\be_0}^+$, so neither $\hyp$ nor $H_{\be_0}^+$ forms a facet of $\padm$. 
	
	When $ \dim \hyp \cap \padm < \dim H_{\be_0}^+ \cap \padm,$ from Lemma \ref{lem: simple}, we get that $ \dim \hyp \cap \padm < N-2$, where $N$ is the dimension of $\mathrm{T}_m \mathfrak{D}(\Pi)$.
So $\hyp$ does not form a facet of $\adm$.
\end{proof}

\section{Fully decorated surfaces} \label{sec:fully}
In the fully decorated cases, the admissible cone is a properly convex projective polyhedron, i.e.\ a convex polyhedron that is contained and bounded in some affine chart. 
We show that the vertices of these polytopes correspond to spike-to-edge arcs. 
We also show that for $n$ large, all the vertices are non-simple, i.e.\ have more incident facets than the ambient dimension.
\subsection{Fully decorated ideal polygons}
\begin{lemma}\label{lem: vip}
    Let $\Pi=\dep n$, for $(n\geq 3)$ and $m\in \tei \Pi$.
Then all the vertices of the projectivised admissible cone $\padm$ are given by $[f_\al(m)]$, where $\al$ is a spike-to-edge arc.  
\end{lemma}

\begin{figure}[ht!]
    \centering
    \includegraphics[width=0.7\linewidth]{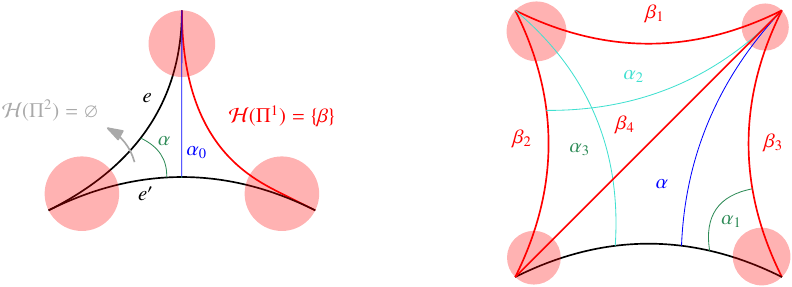}
    \caption{Left-panel: Any horoball connection disjoint from an edge-to-edge arc is disjoint from a spike-to-edge arc.
Right panel: The spike-to-edge arc $\al$ is the only arc disjoint from the horoball connections $\be_1,\ldots, \beta_4$. }
    \label{fig: e2e}
\end{figure}
\begin{proof}
Any vertex of $\padm$ is given by $[f_\al(m)]$ where $\al$ is either spike-to-edge or edge-to-edge.
Firstly, we prove that the second case is not possible.
Let $\al \in \partial\ac {\dep n}$ be an edge-to-edge arc, 
separating the surface into $\Pi^1, \Pi^2$.
Let $\mathcal{H}_{f_{\al}}$ be the maximal family of hyperplanes containing $[f_\al(m)]$ and defined by horoball connections:
\[ \mathcal{H}_{f_{\al}}=\{H_\be\mid \be \in \mathcal{H}(\Pi^1)\sqcup \mathcal{H}(\Pi^2)
\}, \]
where $\mathcal{H}(\Pi^i)$ is the set of all horoball connections contained in the surface $\Pi^i$. 
Without loss of generality, suppose that $\mathcal H(\Pi^1)\neq \emptyset$.
See the left panel in Fig.~\ref{fig: e2e}.
Suppose that $\al$ joins the edges $e, e'$ of $\Pi$. 
Let $\al_0$ be a (permitted) spike-to-edge arc obtained from $\alpha$ by sliding the endpoint that lies in $e$ into the adjacent spike that lies in $\Pi^1$. 
Then any horoball connection that is disjoint from $\al$ is also disjoint from~$\alpha_0$.
Consequently, $\bigcap_{H\in \mathcal{H}_{f_\alpha}} H$ contains at least two distinct points $[f_\al(m)]$, $[f_{\al'}(m)]$; therefore it is not a vertex of $\padm$.

Next we prove that any spike-to-edge arc forms a vertex.
In the right panel of Fig.~\ref{fig: e2e}, the set of all horoball connections from which the spike-to-edge arc $\al$ is disjoint is given by $\{ \be_1,\ldots, \beta_4\}$.
None of the arcs $\al_1,\al_2, \al_3$ are disjoint from all the $\be_i$. 
If possible, let $\al'$ be another arc such that both $[f_\al(m)]$ and $[f_{\al'}(m)] $ belong to $\bigcap_{H\in \mathcal H_{f_\al}} H$.
Letting $F_\alpha$ denote the face of $\padm$ containing $[f_\alpha(m)]$, we may assume that $(\alpha, \alpha')$ is an edge in the simplicial decomposition of $F_\alpha$ provided by Theorem~\ref{thm:mainB}.
Then $\al'$ is disjoint from $\al$ as well as from every $\be$ disjoint from $\al$.
Suppose that $\al'$ is an edge-to-edge arc.
Then at least one of its finite endpoints lies on a boundary horoball connection $\be$ that is different from the one containing the finite endpoint of $\al$.
This $\beta$ is disjoint from $\al$ but not from $\al'$, contradiction: so $\al'$ cannot be of edge-to-edge type.
Next suppose that $\al'$ is of spike-to-edge type.
If the finite endpoints of $\al$ and $\al'$ lie on different boundary components, we can conclude as in the previous case.
Hence, suppose that the finite ends of $\al,\al'$ lie on the same boundary horoball connection $\be$.
So their infinite ends converge to distinct decorated spikes.
In Fig.~\ref{fig: e2e}, this is the case for $\al_3$.
Since $\al_3$ and $\al$ are disjoint isotopy classes, the polygon must have at least four decorated spikes.
Let $v_0$ be a decorated spike lying on the opposite side of $\al'$ as $\al$.
Then the horoball connection ($\beta_4$ in the picture) joining the decorated spike of $\al$ with $v$ intersects $\al'$, while being disjoint from $\al$.
This concludes the proof.
\end{proof}

\begin{lemma}\label{lem: sip}
Let $\Pi=\dep n$ and $m\in \tei \Pi$.
Then for $n\geq 6$, all the vertices of $\padm$ are non-simple.
\end{lemma}
\begin{proof}
    Let $v$ be a vertex of $\padm$. It is non-simple if and only if it is contained in more than $2n-4$ facets of $\padm$. From Lemma~\ref{lem: vip}, we know that $v=[f_{\al}(m)]$, where $\al$ is a spike-to-edge arc. From Lemma~\ref{lem: simple}, we know that every facet is of the form $H_\be\cap \padm$, where $\be$ is a simple horoball connection. We also know that the point $v$ is contained in a facet if and only if there is a $\al$ is disjoint from $\be$. So we need to show that for $n\geq 4$, there are at least $2n-4$ horoball connections that are disjoint from $\al$. 
    
    Now, the arc $\al$ divides the surface into two fully decorated polygons: $\Pi^1:=\dep{k}$ and $\Pi^2:=\dep{n-k+1}$, where $2\leq k\leq n-1$. Then, the total number of horoball connections disjoint from $\al$ is given by the \emph{degree}
    \begin{align*}
    d(k):=|H(\Pi^1)|+|H(\Pi^2)|&=\frac{1}{2}[k(k-1)+(n-k+1)(n-k)]\\
    &=\frac{1}{2}[2k^2+n^2-2k-2nk+n].
    \end{align*}
    Then $d(k)\leq 2n-4$ if and only if $(k,n)=(2,3), (2,4), (3,4), (3,5)$. 
So we conclude that for $n\geq 6$, the total number of horoball connections disjoint from $\al$ is more than $2n-4$, thus proving the lemma.
\end{proof}

\begin{figure}
    \centering
    \includegraphics[width=0.7\linewidth]{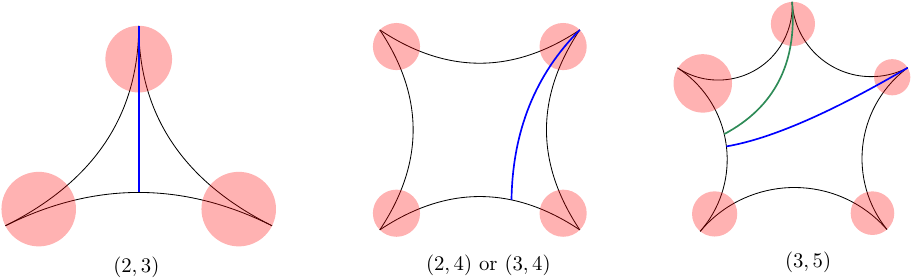}
    \caption{The cases $(k,n)$ for fully decorated ideal polygons where $\padm$ has simple vertices: the arcs corresponding to simple vertices are coloured in blue. In the rightmost figure, the arc corresponding to a non-simple vertex is coloured in green.}
    \label{fig: sip}
\end{figure}
\paragraph{Exceptions.}
The conclusion of Lemma~\ref{lem: sip} fails for $n<6$: in fact for $n=3,4$ \emph{all} vertices are simple, while for $n=5$ the polytope $\padm$ contains a mix of simple and nonsimple vertices.
This is summarized in the table below.
\[
\begin{array}{|c|c|c|c|}
\hline
n & 2n-4 &\#\text{simple vertices} & \#\text{non-simple vertices (degree)}\\
\hline
3 & 2 & 3 & 0 \\
4 & 4 & 4 & 0 \\
5 & 6 & 5 & 10~(7)\\
\hline
\end{array}
\]

\subsection{Once-punctured polygons}
\begin{lemma}\label{lem: vpp}
Let $\Pi= \puncd n$ $(n\geq 2)$ and $m\in \tei \Pi$.
Then all the vertices of the projectivised admissible cone $\padm$ are given by $[f_\al(m)]$, where $\al$ is a spike-to-edge arc.  
\end{lemma}
\begin{figure}[ht!]
    \centering
    \includegraphics[width=0.8\linewidth]{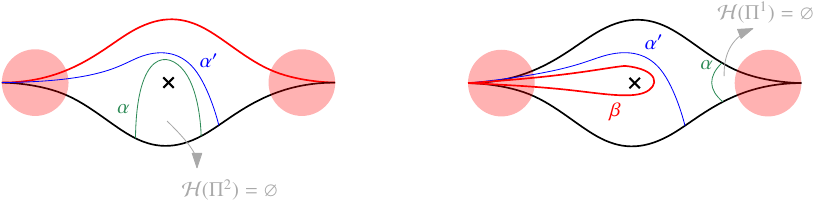}
    \caption{Left panel: The edge-to-edge arc $\al$ separates, from $\puncd 2$, a punctured disk with no horoball connections.
Right panel: The arc $\al$ separates, from $\puncd 2$, a disk with one decorated spike and hence no horoball connections.
In both the cases, the spike-to-edge arc $\al'$ is disjoint from the only horoball connection $\be$ that is disjoint from the edge-to-edge arc $\al$.}
    \label{fig: vpp3}
\end{figure}
\begin{proof}
    Like in the case of ideal polygons, we will firstly show that $[f_\al(m)]$ cannot be a vertex for $\al$ edge-to-edge.
Let $\al \in \partial\ac {\puncd n}$ be an edge-to-edge arc separating the surface into a disk $\Pi^1$ and a punctured disk $\Pi^2$.
Again, let $\mathcal{H}_{f_\al}$ be the maximal family of hyperplanes containing $[f_\al(m)]$ and defined by horoball connections.
Then it is of the form \[
 \mathcal{H}_{f_\al}=\{H_\be\mid \be \in \mathcal H(\Pi^1)\sqcup \mathcal H(\Pi^2)\}.
    \]
    Suppose that $\mathcal H(\Pi^1)\neq \varnothing$.
This is the case in the left panel of Fig.~\ref{fig: vpp3}.
Then $\Pi^1$ has at least 2 decorated spikes.
The spike-to-edge arc that joins one of these two spikes with the edge containing one of the finite endpoints of $\al$ is disjoint from all the horoball connections that are disjoint from $\al$.

\begin{figure}[ht!]
    \centering
    \includegraphics[width=5cm]{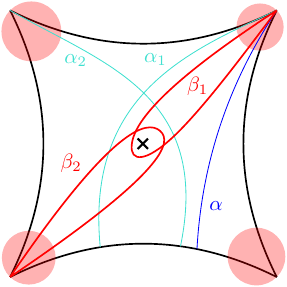}
    \caption{Lemma \ref{lem: vpp} for $\puncd 4$}
    \label{fig: vpp2}
\end{figure}

Now suppose that $\Pi^1$ contains only one decorated spike.
This is illustrated in the right panel of Fig.~\ref{fig: vpp3}.
Let $\al'$ be the spike-to-edge arc obtained from $\alpha$ by sliding one of its endpoints into the opposite puncture, away from $\Pi^1$.
Then $\al'$ is disjoint from all the horoball connections disjoint from $\al$.
Consequently, both $[f_\al(m)]$ and $[f_{\al'}(m)] $ belong to $\bigcap_{H\in \mathcal{H}_{f_\alpha}} H$, which is therefore not a vertex of $\padm$.

Next we prove that $[f_\al(m)]$ forms a vertex for any spike-to-edge arc $\al$ joining a decorated spike $v$ and a boundary edge $e$.
If possible, let $\al'$ be another arc such that the segment joining the two points $[f_\al(m)],[f_{\al'}(m)] $ lies in $\bigcap_{H\in \mathcal H_{f_\al}} H$.
Using the arguments from the proof of Lemma \ref{lem: vip}, we have that $\al'$ cannot be an edge-to-edge arc nor a spike-to-edge arc whose finite endpoint lies on a boundary edge different from $e$.
Let $\al'$ be a spike-to-edge arc joining $v'$ and $e$ where it is possible to have $v=v'$.
Suppose that $\al' \subset \Pi^1$.
Then using the arguments in Lemma \ref{lem: vip} we get that $\al'$ cannot be disjoint from all the horoballs in $H(\Pi^1)$.
Suppose that $\al' \subset \Pi^2$.
This is illustrated in Fig.~\ref{fig: vpp2}.
Let us assume that $\al'$ separates the puncture from $\al$.
This is the case for the arc $\al_2$ in Fig.~\ref{fig: vpp2}.
Then the infinite end of $\al'$ cannot converge to $v$.
So the maximal horoball connection based at $v$ intersects $\al'$ while being disjoint from $\al$.
So $\al'$ separates a disk from $\al$.
This is the case illustrated by $\al_1$ in Fig.~\ref{fig: vpp2}.
Since the surface is fully decorated, the disk has at least one decorated spike, say $v'$.
Then the maximal horoball connection based at $v'$ intersects $\al'$ while being disjoint from $\al$.
Hence $[f_{\al'}(m)] \notin\bigcap_{W\in \mathcal H_{f_\al}} W$.
So $[f_{\al}(m)]$ is a vertex of $\padm$.
\end{proof}

\begin{lemma}\label{lem:spp}
    Let $\Pi=\puncd n$ and $m\in \tei \Pi$.
Then for $n\geq 4$, the vertices of $\padm$ are non-simple.
\end{lemma}
\begin{proof}
    Like in Lemma \ref{lem: sip}, we need to show that for $n\geq 4$, there are more than $2n-2$ horoball connections that are disjoint from every spike-to-edge arc $\al$. Suppose that $\al$ decomposes the surface into two subsurfaces $\Pi^1:=\puncd k$ and $\Pi^2:=\dep {n-k+1}$.
    Let $d(k)$ be the degree of $\al$. Then we have that 
     \begin{align*}
    d(k):=|H(\Pi^1)|+|H(\Pi^2)|&=\frac{1}{2}[(n-k+1)(n-k)]+k^2\\
    &=\frac{1}{2}[3k^2+n^2-k-2nk+n].
    \end{align*}
Then $d(k)\leq 2n-2$ if and only if $(k,n)=(1,2), (1,3)$. Thus for $n\geq 4$, all the vertices of $\padm$ are non-simple.
    \end{proof}
\paragraph{Exceptions.}   
The conclusion of Lemma~\ref{lem:spp} fails for $n<4$: in fact for $n=1,2$ \emph{all} vertices are simple, while for $n=3$ the polytope $\padm$ contains a mix of simple and nonsimple vertices.
This is summarized in the table below.
\[
\begin{array}{|c|c|c|c|}
\hline
n & 2n-2 &\#\text{simple vertices} & \#\text{non-simple vertices (degree)}\\
\hline
1 & 0 & 1 & 0  \\
2 & 2 & 4 & 0  \\
3 & 4 & 6 & 6~(5) \\
\hline
\end{array}
\]
\subsection{Crowns}
\begin{lemma}\label{lem: vcr}
    Let $\Pi=\dholed n$ and $m\in \tei \Pi$.
When $n=1$, the vertices of $\padm$ are given by one spike-to-edge arc and one edge-to-edge arc.
For $n\geq 2$, all the vertices of $\padm$ are given by $[f_\al(m)]$, where $\al$ is a spike-to-edge arc.  
\end{lemma}
\begin{figure}[ht!]
    \centering
    \includegraphics[width=0.8\linewidth]{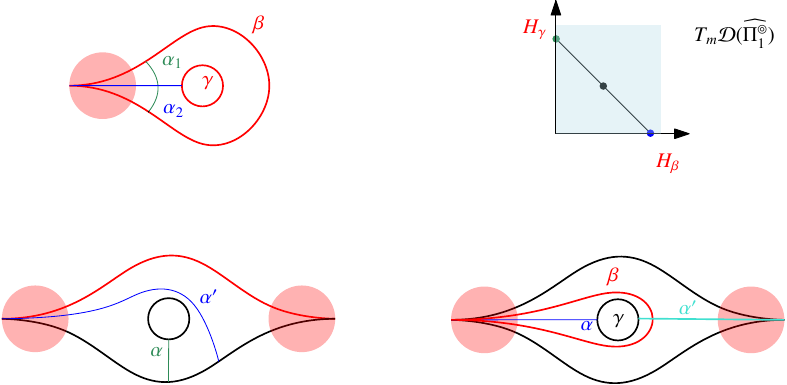}
    \caption{Lemma \ref{lem: vcr}}
    \label{fig: vcr}
\end{figure}

\begin{proof}
    The case $n=1$ is illustrated in Fig.~\ref{fig: vcr}.
The edge-to-edge arc $\al_1$ is the only arc that is disjoint from $\ga$.
Similarly, the spike-to-edge is the only arc that is disjoint from the horoball connection $\beta$.
The projectivised admissible cone is of dimension 1.
This is illustrated in the top right panel.
The points $[f_{\al_1}(m)], [f_{\al_2}(m)] $ form the two vertices of $\padm$.

Now we suppose that $n\geq 2$.
Suppose that $\al \in \partial\ac {\dholed n}$ is an edge-to-edge arc.
From the proof of Lemma \ref{lem: vpp}, we know that $\al$ cannot be a separating arc.
So $\al$ joins the geodesic boundary $\ga$ with a boundary horoball connection $e$.
Since $n\geq 2$, $[f_\al(m)]$ is not an internal point of $\padm$.
Let $v$ be an endpoint of $e$.
Since $n\geq 2$, the spike-to-edge arc $\al'$ joining $v$ with $e$ is disjoint from $\al$ as well as all horoball connections that are disjoint from $\al$.
This is illustrated in the bottom left panel of Fig.~\ref{fig: vcr}.
So $[f_\al(m)]$ is not a vertex. 

Next we prove that for any spike-to-edge arc $\al$, the point $[f_\al(m)]$ forms a vertex.
Using the arguments in the proof of Lemma \ref{lem: vpp}, we already know that $[f_\al(m)]$ forms a vertex for any separating spike-to-edge arc $\al$.
So suppose that $\al$ is non-separating spike-to-edge arc.
So it joins a decorated spike $v$ with $\ga$.
See bottom right panel in Fig.~\ref{fig: vcr}.
If possible, let $\al'$ be another arc that is disjoint from $\al$ as well as all the horoball connections disjoint from $\al$.
We only treat the case when $\al'$ is a spike-to-edge arc joining a decorated spike $v'$ with $\ga$.
Since $\al$ is the only arc that joins $v$ to $\ga$, we have that $v\neq v'$.
Then the maximal horoball connection $\beta$ based at $v$ is disjoint from $\al$ but intersects $\al'$.
Hence $[f_{\al'}(m)]$ does not belong to $\bigcap_{W\in H(\al)} W$.
So $[f_{\al}(m)]$ is a vertex of $\padm$.
\end{proof}

\begin{lemma}
    For $n\geq 4$, all the vertices of $\padm$ of a fully decorated crown $\dholed n$ are non-simple.
\end{lemma}
\begin{proof}
    Let $\al$ be a spike-to-edge arc. 
    The case where $\al$ is separating is already treated in Lemma~\ref{lem:spp}. Suppose that $\al$ is non-separating. Cutting along $\al$ we get a fully decorated ideal polygon with $n+1$ spikes. Then the degree of the vertex $[f_{\al}(m)]$ is given by $ d(\al):=\frac{n(n+1)}{2}$. Then we have
    \begin{align*}
        d(\al) \leq 2n-1& \Leftrightarrow n^2-3n+2\leq 0\\
        & \Leftrightarrow (n-1)(n-2) \leq 0.
    \end{align*}
    So for $n\geq 3$, $d(\al)> 2n-1$.
    \end{proof}

\paragraph{Exceptions.} From Lemma~\ref{lem:spp}, we get that the vertex $[f_{\al}(m)]$, corresponding to a separating spike-to-edge arc, is simple for the cases $(k,n)=(1,2), (1,3)$. For $n=1$, there is one separating simple vertex given by an edge-to-ege arc and one non-separating simple vertex given by a spike-to-edge arc. See top left panel of Fig.\ref{fig: vcr}. When $n=1,2$, every non-separating spike-to-edge arc is a simple vertex. In the case $n=3$, every non-simple vertex is separating. 

\[
\begin{array}{|c|c|c|c|c|}
\hline
n & 2n-1 &\#\text{simple vertices:} & \#\text{simple vertices:} &\#\text{non-simple vertices (degree)}\\
&&  \text{separating} & \text{non-separating} &\\
\hline
1 & 1 & 1 & 1 & 0 \\
2 & 3 & 4 & 2 & 0  \\
3 & 5 & 6 & 0 & 6~(5) \text{ (sep.) }\\
\hline
\end{array}
\]

\subsection{M\"obius strip with spikes}
\begin{lemma}\label{lem: vmb}
Let $\Pi=\mobd n$ and $m\in \tei \Pi$.
When $n=1$, the vertices of $\padm$ are given by one spike-to-edge arc and one edge-to-edge arc.
For $n\geq 2$, all the vertices of $\padm$ are given by $[f_\al(m)]$, where $\al$ is a spike-to-edge arc. 
\end{lemma}
\begin{figure}[ht!]
    \centering
    \includegraphics[width=0.5\linewidth]{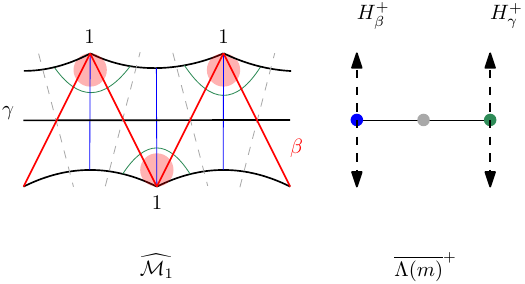}
    \caption{The vertices of $\padm$ for $\mobd 1$ are given by one spike-to-edge arc (in blue) and one edge-to-edge arc (in green).}
    \label{fig: vmb}
\end{figure}

\begin{proof}
The case $n=1$ is illustrated in Fig. \ref{fig: vmb}. There are three permitted arcs. The only spike-to-edge arc, say $\al_1$ is coloured blue. It is disjoint from the non-separating horoball connection $\beta$. The green arc, say $\al_2$, is an edge-to-edge arc disjoint from the core curve $\gamma$. The vertices of the one-dimensional polytope $\padm$ are given by $[f_{\al_1}(m)],[f_{\al_2}(m)]$. 

\begin{figure}[h!]
    \centering
    \includegraphics[width=0.5\linewidth]{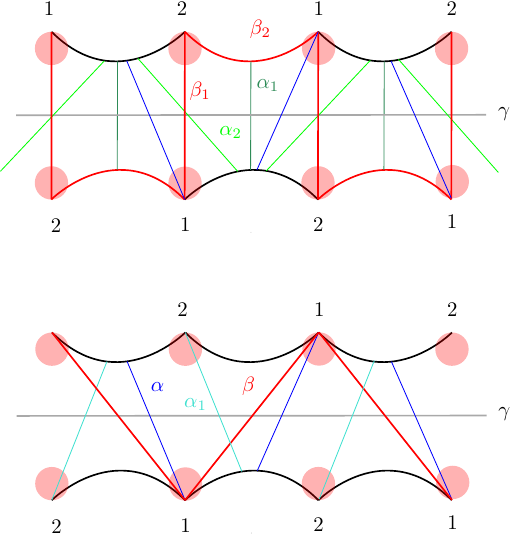}
    \caption{Lemma \ref{lem: vmb}}
    \label{fig: vmb2}
\end{figure}
Next we prove that for $n\geq 2$, the point $[f_{\al}(m)]$ cannot be a vertex of $\padm$ when $\al$ is an edge-to-edge arc. From the proof of Lemma \ref{lem: vpp}, we know that $\al$ cannot be a separating arc.
Since $n\geq 2$, $[f_\al(m)]$ is not an internal point of $\padm$.
Suppose that the endpoints of $\al$ lie on the boundary edges $e,e'$, not necessarily distinct. See Fig.~\ref{fig: vmb2}. Let $v$ be an endpoint of $e$.
Since $n\geq 2$, the spike-to-edge arc $\al'$ joining $v$ with $e$ is disjoint from $\al$ as well as all horoball connections that are disjoint from $\al$.

Next we prove that for any spike-to-edge arc $\al$, the point $[f_\al(m)]$ forms a vertex.
Using the arguments in the proof of Lemma \ref{lem: vpp}, we already know that $[f_\al(m)]$ forms a vertex for any separating spike-to-edge arc $\al$.
So suppose that $\al$ is non-separating spike-to-edge arc.
See bottom panel in Fig.~\ref{fig: vmb2}.
If possible, let $\al_1$ be another arc that is disjoint from $\al$ as well as all the horoball connections disjoint from $\al$.
We only treat the case when $\al_1$ is a non-separating spike-to-edge arc joining a decorated spike $v_1$ with $\ga$.
Since $\al$ is the only non-separating arc that joins $v$ to $\ga$, we have that $v\neq v'$. In the bottom panel, we have $v=1, v_1=2$. 
Then the maximal horoball connection $\beta$ based at $v$ is disjoint from $\al$ but intersects $\al_1$. 
Hence $[f_{\al'}(m)]\notin \bigcap_{W\in \mathcal H_{f_\al}} W$.
So $[f_{\al}(m)]$ is a vertex of $\padm$.
\end{proof}
\begin{lemma}
    For $n\geq 4$, all the vertices of $\padm$ of a fully decorated Möbius strip $\mobd n$ are non-simple.
\end{lemma}
\begin{proof}
    Let $\al$ be a spike-to-edge arc. 
    The case where $\al$ is separating is already treated in Lemma~\ref{lem:spp}. The remaining case where $\al$ is non-separating is identical to Lemma \ref{lem: vcr}. 
    \end{proof}

 The exceptions are same as in the case of fully decorated crowns.

\newpage

\appendix
\begin{appendices}
\section{Topology of arc complexes}\label{sec:app1}
Here we review some auxiliary material, especially the ballness results used in Section~\ref{sec:mainproofs} to prove Proposition~\ref{prop:dimofmac}. 
Depending on the type of surface some of these proofs are scattered in the literature, but we provide a unified treatment here.

\paragraph{Topological surfaces.} For $n\geq 1$, we denote by $\poly n$ a closed disk with $n\geq 1$ marked points on its boundary. For $n\geq 3$, this surface, when endowed with a convex Euclidean metric, becomes the usual convex polygon. 
For $n\geq 1$, we denote by $\punc n$, a closed once-punctured disk with $n\geq 1$ marked points on its boundary.
For $n\geq 1$, we denote by $\cro n$, the closed disk $\poly n$ with one marked point in its interior labeled as $0$. We will refer to this topological surface $\cro n$ as an orientable \emph{crown}.
For $n\geq 1$, we denote by $\mob n$ the M\"{o}bius strip with $n$ marked point on its boundary. We will refer to this surface as the \emph{non-orientable crown}. 

Given a surface $S$ of the four above types, we consider all possible colourings of its marked points with two colours, say red $(R)$ and blue $(B)$, so that there is a vertex of each colour. Furthermore, in the case of orientable crowns, the internal marked point is always colored blue. Such a colouring is called a \emph{bicoloring}. A bicoloring is called \emph{non-trivial} if there is at least one $R-R$ diagonal and at least one $R-B$ diagonal. 
We denote by $\cac S$ the subcomplex of the arc complex $\ac S$ spanned by the $R-B$ and $B-B$ diagonals only. We call these diagonals \emph{permitted} and the $R-R$ diagonals are called \emph{rejected}. A simplex of $\cac S$ using permitted diagonals is called \emph{permissible}. 

\begin{definition}\label{def: boundary}
In an orientable crown $\cro n$, any arc that does not have an endpoint on the vertex $0$, is called a \emph{boundary arc} or a b-arc in short. In a non-orientable crown $\mob n$, any arc that decomposes the strip into one orientable and one non-orientable subsurface is called a boundary arc or b-arc. In either case, a non-boundary arc is called \emph{core arc} or simply \emph{c-arc}.
\end{definition}
\begin{remark}
A c-arc of a non-orientable crown $\mob n$ always intersects its one-sided core curve. Similarly, any c-arc of a one-holed polygon $\cro n$ intersects the two-sided core curve of $\cro n \smallsetminus \{0\}$.
\end{remark}

%

\begin{notation}\label{notation}
\begin{itemize}
\item A b-arc of a (resp. non-orientable) crown $\cro n$ (resp.\ $\mob n$) decomposes the surface into two tiles: one homeomorphic to a polygon and one homeomorphic to a (resp. non-orientable) crown with at most $n+1$ vertices. Given two vertices $i,j$ such that $1\leq i\leq j \leq n$, there are at most two b-arcs joining them. When $j-i>1$, let $a_i^j$ be the b-arc  whose polygonal tile is homeomorphic to $\poly {j-i+1}$, containing the vertices $i, i+1,\ldots,j$; its non-polygonal tile is homeomorphic to $\cro{n-j+i+1}$ (resp.$\mob{n-j+i+1}$), containing the vertices $j+1,\ldots, n, 1\ldots, i$. When $j-i<n-1$, we denote the second b-arc by $a_j^i$. 
	\item Any c-arc of the surface $\cro n$ joins the internal point $0$ with some vertex $i\in \intbra$ in the boundary. So these arcs are going to be denoted by $c_i$. 
	\item For $i\in \intbra$ and $S=\cro n, \mob n$,  we denote by $M_i$ the maximal b-arc with both its endpoints at the vertex labeled $i$. 
	\item For $i\in \intbra$ and $S= \mob n$, we denote by $L_i$ the maximal c-arc with both its endpoints at the vertex labeled $i$. 
	\item Given two vertices $i,j$ of $\mob n$ such that $1\leq i< j\leq n$, we denote by $(i,j)$ the unique c-arc joining these two vertices. 
\item An arc of $S=\poly n,\cro n, \mob n$ is called \emph{minimal} if it separates a disk homeomorphic to triangle $\poly 3$ from the surface. 
\item The subcomplex $\yacross n$ of $\cac{\mob n}$ generated by permitted c-arcs only is called the permitted internal arc complex. 
\item In the orientable crown $\cro n$, let $\Delta_{n-1}$ be the $(n-1)$-dimensional simplex of $\ac{\cro n}$ generated by all the arcs that join one boundary marked point to the blue marked point in the interior. This simplex is always present in $\cac{\cro n}$, irrespective of the bicoloring.
\end{itemize}
\end{notation}
	\begin{definition}
		Let $\sigma$ be a simplex of the arc complex $\ac S$ of a surface $S$. Let $a_1,\ldots, a_k$ be pairwise disjoint and distinct arcs such that the 0-simplices of $\sigma$ are given by $\sigma^{(0)}=\bigcup\limits_{i=1}^k\{ [a_i]  \}$.
		Then we say that $\sigma$ decomposes the surface $S$ into the \emph{tiles} $\del_1,\ldots, \del_p$ if $$S\setminus \bigcup\limits_{i=1}^k a_i=\del_1\sqcup\ldots\sqcup\del_p,$$ where $\del_1,\ldots, \del_p$ are the connected components. 
\end{definition} 

\subsection{Previous results}
A simplicial complex $X$ is called a \emph{combinatorial $d$-manifold with boundary} if the link of each 0-simplex is either a combinatorial $(d-1)$-sphere or a combinatorial $d-1$-ball and there exists a 0-simplex such that its link is of the latter kind. 

\paragraph{Collapses}Let $X$ be a simplicial complex and $\sigma$ a simplex of $X$. The \emph{face-deletion} of $\sigma$ in $X$ is the subcomplex defined as $\fdel \sigma X:=\{\eta\in X \mid \sigma\nsubseteq \eta\}$. In particular, the face deletion of a 0-simplex $v$ in $X$ is denoted by $X\smallsetminus \{v\}$. 
Let $\sigma\subsetneq\tau \in X$ be two simplices  such that $\tau$ is the only maximal simplex containing $\sigma$.  Then $\sigma$ is called a \emph{free face} of $X$. The tuple $(\sigma, \tau )$ is called a \emph{collapsible pair}. The complex $X$ is said to \emph{simplicially collapse} onto its  subcomplex $\fdel \sigma X$, if $\sigma$ is a free simplex of $X$. The complex $X$ is said to be \emph{collapsible} if there is a finite sequence of simplicial collapses leading to a 0-simplex. By $X\searrow Y$, we mean that the complex $X$ simplicially collapses onto the complex $Y$. 
Whitehead \cite{whitehead} showed that if $X, Y$ are two simplicial complexes such that $X\searrow Y$, then $X$ has the same homotopy type as $Y$. A 0-simplex $v\in X$ is said to be \emph{vertex-dominated} if there exists another 0-simplex $v'\in X$ such that $\Link{v}{X}=v' \Join L,$ where $L$ is a subcomplex of $X$. In other words, any maximal simplex containing $v$ must also contain $v'$. A complex $X$ is said to \emph{strongly collapse} on $ X\smallsetminus \{v\}$ if the 0-simplex $v$ is vertex-dominated in $X$. We will denote this operation as $X\Searrow  X\smallsetminus \{v\}$.

Next we list results that will be used in the proofs:
 \begin{theorem}\label{thm: discol}
The arc complex of a convex polygon $\poly m$ ($m\geq 4$) is PL-homemorphic to a sphere of dimension $m-4$. The arc complex of a once-punctured polygon $\poly m$ ($m\geq 4$) is PL-homemorphic to a sphere of dimension $m-4$. 
 \end{theorem}
\begin{theorem}\label{thm: bicolpoly}
	The subcomplex $\cac{\poly m}$ of a polygon $\poly m$ with any non-trivial $R-B$ bicoloring is a closed ball of dimension $m-4$. Similarly, the subcomplex $\cac{\punc m}$ of a once-punctured polygon $\punc m$ with any non-trivial bicoloring is a closed ball of dimension $m-2$.
\end{theorem}
\begin{theorem}\label{thm: crown}
    For $n\geq 1$, the full arc complexes $\ac{\cro n}$ and $\ac{\mob n}$ of an orientable and a non-orientable crown respectively are combinatorial balls of dimension $n-1$.
\end{theorem}
Barmak-Minian \cite{strong} proved the following theorems about strong collapsibility which will be used later.
\begin{theorem}\label{thm: barmak}
	\begin{enumerate}[label=\alph*)]
	\item \label{simultcoll}
		Let $L$ be a subcomplex of a complex $K$ such that every vertex of $K$ which is not in $L$ is dominated by some vertex in $L$. Then $K\Searrow L$.
	\item \label{strong2simple}
	If $X,Y$ are two simplicial complexes such that $X\Searrow Y$, then $X\searrow Y$.
	\item \label{scjoin}
	Given two simplicial complexes $X,Y$, their join $X\Join Y$ is strongly collapsible if and only if either $X$ or $Y$ is strongly collapsible.
	\end{enumerate}
\end{theorem}

The following theorem was proved by Whitehead \cite{whitehead}.
\begin{theorem}\label{thm: whitehead}
A collapsible combinatorial $d$-manifold is a combinatorial $d$-ball.
\end{theorem}

\subsection{Bicolored crown}
In this section, we prove that the permitted arc complex of a bicolored crown is a closed ball.
\begin{theorem}\label{strcollhole}
	For $n\geq 1$, the subcomplex $\cac{\cro n}$ of a bicolored crown $\cro n$ is strongly collapsible.
\end{theorem}	
\begin{proof}
	We will prove that the complex $\cac{\cro n}$ strongly collapses onto the simplex $\Delta_{n-1}$. 

	For $k\in \intbra$, let $\mathcal{B}_k$ be the set of all permitted arcs that decompose the surface into two subsurfaces homeomorphic to $\cro {k}$ and $\poly {n-k+2}$, respectively. Define the following subcomplexes of $\ac{\cro n}$:
	\[ 
	Y^k:=\left\{
	\begin{array}{ll}
	\ac{\cro n}, & \text{ when } k=0,\\
	 Y^{k-1}\smallsetminus \{[a] \mid a\in \mathcal{B}_k
 \}& \text{ when } k \in \intbra[1,n-1] \\
	\end{array}
	\right.
	\]
    Then we have the following lemma:
	\begin{lemma}
		For $k=0,\ldots, n-2$, the subcomplex $Y^k$ strongly collapses onto the subcomplex $Y^{k+1}$.
	\end{lemma}
    The set $\mathcal{B}_1$ is empty if and only if all the marked points on the boundary are red. In this case the subcomplex $\cac{\cro n}=\Delta_{n-1}$, which is strongly collapsible. So we suppose that for every $k$, he set $\mathcal{B}_k$ is non-empty.
	Now prove this claim by induction on $k$. \\
	\emph{Base step:} The set $\mathcal{B}_1$ comprises of all the permitted maximal arcs of $\cro n$. 
	Indeed, any maximal arc $M_i$ decomposes the surface $\cro n$ into a subsurface homeomorphic to $\poly{n+1}$ and a subsurface homeomorphic to $\cro 1$, containing the vertex $0$ in its interior and the vertex $i$ in its boundary. So let us suppose that $M_i\in \mathcal{B}_1$ is a permitted maximal arc for some blue vertex $i$.
    
    The link of the 0-simplex $[M_i]$ in the full arc complex is given by 
	\[
	\Link {[M_i]}{Y_0} = \ac{\poly{n+1}} \Join [c_i],\\
	\]
    where $c_i$ is the permitted arc joining the vertex $i$ with the internal vertex 0.
	This shows that whenever $M_i$ is permitted for some $i\in\intbra$, the 0-simplex $[M_i]$ is vertex-dominated by $[c_i]$. 
	Using Theorem \ref{thm: barmak}\ref{simultcoll}, we get that $Y^0\Searrow Y^1$. This finishes the base step $k=0$.\\
	\emph{Induction step:} Suppose that for $k'=0,\ldots, k-1$, we have $Y^{k'}\Searrow Y^{k'+1}$. We need to show that $Y^{k} \Searrow Y^{k+1}$. 
	Let $a\in \mathcal{B}_{k+1}$ be a permitted arc with endpoints on the vertices $i\neq j$. Then following Notation \ref{notation}, we have that $a=a_j^i$. We claim that the 0-simplex $[a]$ is vertex-dominated by the permitted $[c_i]$. Any permitted arc that is disjoint from $a$ but not with the arc $c_{i}$ is a b-arc $\aij {j'}{i'}$ contained in $\cro{k+1}$. Hence, $i<i'\leq j'\leq j$. See Fig. \ref{fig: bicrown}. Such an arc $\aij {j}{i+1}$  lies in $\mathcal{B}_k$. The 0-simplex $[\aij {j}{i+1} ]$ is absent from $Y^k$, by induction hypothesis. Thus every arc in $\mathcal{B}_{k+1}$ is vertex dominated in $Y^k$. 
Using Theorem \ref{thm: barmak}\ref{simultcoll}, as in the base step, we conclude that $Y^k \Searrow Y^{k+1}$.
Since any boundary arc lies in $\mathcal{B}_k$ for some $k$, we get that $Y^{n-1}= \Delta_{n-1}$. This concludes the proof.
	\end{proof}

\begin{figure}[ht!]
    \centering
    \includegraphics[width=0.5\linewidth]{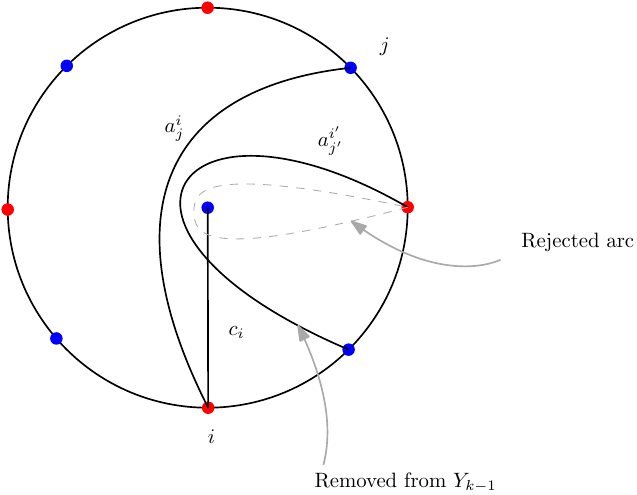}
    \caption{}
    \label{fig: bicrown}
\end{figure}

\begin{cor*}\label{acrownball}
For $n\geq1$, the arc complex of an orientable crown $\cro n$ is a combinatorial ball of dimension $n-1$. 
\end{cor*}

\begin{proof}
Let $\sigma$ be any simplex of $\cac{\cro n}$. If there are no b-arcs in $\sigma$, then it contains only c-arcs that join the interior marked point with some marked points on the boundary. In this case, the simplex $\sigma$ can be completed into the unique triangulation containing only such type of arcs. Now suppose that $\sigma$ contains at least one b-arc. It divides the surface into a bicolored polygon $P^1$ homeomorphic to $\poly k$ and one crown $P^2$ homeomorphic to $\cro{n-k+2}$.  Choose the b-arc, say $b$, such that the crown in its complement has no b-arc. Since, $b$ is permitted, at least one one its endpoints is blue. So both $P^1$ and $P^2$ have at least one blue vertex. A maximal simplex of $\cac{\cro n}$ contains $\sigma$ if and only if $\sigma$ restricted to $P^1, P^2$ are maximal simplices $\sigma_1,\sigma_2$ of $\cac{P^1}, \cac{P^2}$, respectively. So we have that 
\begin{align*}
    \dim \sigma &= \dim \sigma_1+\dim\sigma_2+2\\
    &= k-4 + (n-k+2-1)+2\\
    &=n-1.
\end{align*}
This shows that every maximal simplex of $\cac{\cro n}$ has dimension $n-1$.

Next, we show that every $(n-2)$-simplex $\sigma$ of $\cac {\cro n}$ is contained in at most two $(n-1)$-simplices of $\ac {\cro n}$. Firstly, suppose that $\sigma $ is a codimension 1 boundary simplex of $\cac{\cro n}$ as well as $\ac{\cro n}$. Then it decomposes the surface into triangles and a crown $\mob 1$ with only one vertex, say $i\in \intbra$. The only way to triangulate this tile is to take the c-arc $c_i$. So $\sigma$ is contained in the unique maximal simplex $\sigma\cup \{[c_i]\}$. Next, we suppose that $\sigma\in \cac{\cro n}$ contains an internal arc $c$. Cutting the surface along this arc we get a surface homeomorphic to a non-trivially bicolored convex $n+2$-gon. The restriction of $\sigma$ on this polygon gives a $(n-3)$-simplex, say $\sigma'$, of $\cac{\poly {n+2}}$. Since the latter is a $(n-2)$-pseudo-manifold with boundary (\cite{ppballs}), the restriction of $\sigma'$ is contained in at most two $(n-2)$-simplices of $\ac{\poly{n+1}}$. Taking the join of these two maximal simplices with $[c]$ we get that $\sigma$ is contained in exactly two $(n-1)$-simplices of $\cac {\cro n}$. So we get that the simplicial complex $\ac {\cro n}$ is a pseudo-manifold. \\
\begin{figure}
    \centering
    \includegraphics[width=0.5\linewidth]{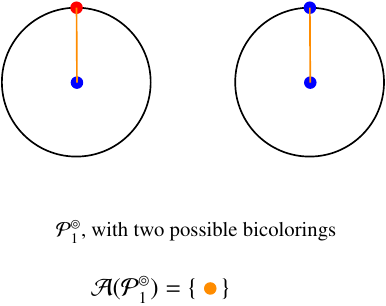}
    \caption{The arc complex of $\cro 1$ with any bicolouring is a ball of dimension zero.}
    \label{fig:acrown}
\end{figure}
Now, we prove that the arc complex is a combinatorial ball of dimension $n-1$, by induction. From Fig. \ref{fig:acrown}, we know that the statement is true for $n=1$. Suppose that the statement is true for $n'=1,\ldots, n-1$. 
Let $\al$ be a c-arc. Then it decomposes the surface into a $\poly {n+2}$. Its link is then given by
\begin{align*}
\Link{[\al]} {\cac{\cro n}}=\cac {\poly{n+2}}&= \left\{
\begin{array}{rl}
 \s{n-2}   & \text{ when $\poly{n+2}$ has a trivial bicoloring,} \\
\ball{n-2} & \text{ otherwise.}
\end{array}
\right.
\end{align*}
Next, suppose that $\al$ is a permissible b-arc of the form $a_i^j$ with $j-i\geq2$. Then it decomposes the surface into a bicolored polygon $\poly{j-i+1}$ and a bicolored crown $\cro{n-j+i +1}$. Then its link is given by
\begin{align*}
\Link{[\al]} {\cac{\cro n}}&=\cac{\poly{j-i+1}}\Join \cac{\cro{n-j+i+1}}\\
&=\left\{
\begin{array}{rl}
 \s{j-i+1-4}\Join \ball{n-j+i+1-1 }    & \text{ when $\poly{j-i+1}$ has a trivial bicoloring,} \\
   \ball{j-i+1-4}\Join \ball{n-j+i+1-1 }  & \text{ otherwise.}
\end{array}
\right.\\
&=\ball{n-2}.
\end{align*}
So we get that $\cac{\cro n}$ is a combinatorial manifold of dimension $n-1$. Finally, using Whitehead's Theorem \cite{whitehead} and our Theorem \ref{strcollhole}, we get that the full arc complex is a combinatorial $(n-1)$-ball.
\end{proof}

\subsection{Bicolored non-orientable crown}
Next we prove the ballness of the permitted arc complex of a bicolored $\mob n$.
\begin{theorem}
For $n\geq 2$, the subcomplex $\cac{\mob n}$ of the arc complex $\ac{\mob n}$ of a non-trivially bicolored non-orientable crown $\mob n$ is a combinatorial ball of dimension $n-1$.
\end{theorem}

In order to prove the above theorem, firstly we show that,
\begin{theorem}\label{thm: strcoll}
	For $n\geq 2$ and every non-trivial bicoloring, the permitted inner arc complex $\yacross n$ of a non-orientable crown $\mob n$ is strongly collapsible.
\end{theorem}
	\begin{proof}
We prove the statement by induction on the number of vertices $n$.
When $n=2$, there is only one non-trivial colouring, given by one blue vertex, say $1$ and one red vertex at $2$. In this case, the inner arc complex is the 1-simplex generated by $[L_1]$ and $(1,2)$, which is strongly collapsible.

Next we suppose that the statement holds for all $n'\in \intbra$ and for all non-trivial bicolorings. Given the surface $\mob n$, without loss of generality, we add a new vertex labeled $n+1$ between the vertices 1 and $n$ to get the surface $\mob {n+1}$. If $\mob {n}$ has no blue vertices, then the vertex labeled $n+1$ has to be blue in order to make the bicoloring of $\mob{n+1}$ non-trivial. In this case the permitted inner arc complex $\yacross{n+1}$ is a $(n-1)$-simplex generated by the permitted c-arcs $a_i:= [(i,n+1)]$, for $i\in \intbra[1,n+1]$. So it is strongly collapsible, thus finishing the proof. So let us suppose that $\mob n$ has at least one blue vertex. Then the bicoloring of $\mob{n+1}$ is nontrivial in both the cases where the vertex $n+1$ is red and blue. By induction hypothesis, $\yacross{n}$ is strongly collapsible. We prove that the complex $\yacross{n+1}$ strongly collapses onto the complex $\yacross{n} $ by removing all new permitted arcs that arise in $\mob{n+1}$ due to the addition of the new vertex. 

\begin{figure}
    \centering
    \includegraphics[width=0.5\linewidth] {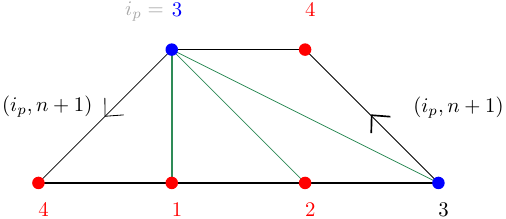}
    \caption{Step 1 of Theorem \ref{thm: strcoll}: $n=3,\, i_p=3$. Any arc disjoint from $(i_p, n+1)$ (coloured in green) is also disjoint from $(1, i_p)$.}
    \label{fig: step1}
\end{figure}
Suppose that the arcs $\{(i_k,n+1)\}_{k=1}^p$ are permitted for $1\leq\ldots<i_1<i_2<\ldots<i_p\leq n+1,$ $ k\in \intbra[1,n+1]$. 

\begin{align*}
\yacross{n+1}  &\Searrow \yacross{n+1}   \smallsetminus \{a_{i_p} \} \\
&\Searrow \yacross{n+1}   \smallsetminus \{a_{i_p}, a_{i_{p-1}} \} \\
& \Searrow \ldots \\
&\Searrow  \yacross{n+1}   \smallsetminus \{a_{i_p}, a_{i_{p-1}},\ldots, a_{i_1}\} \\
&=\yacross{n} .
\end{align*}
		
\begin{step} 
$\across{n+1} \Searrow  \across{n+1}   \smallsetminus \{a_{p} \}$: If the vertex $n+1$ is blue, then $i_p=n+1$ and $a_{i_p}=[L_{n+1}]$. This 0-simplex is contained in a unique maximal simplex, namely, $\ffan [n+1]$. So it is vertex-dominated by $a_i$ for any permissible $i\in \intbra$. Such an $i$ exists because because we are in the case where the initial surface $\mob n$ had a non-trivial bicoloring before the addition of the vertex $n+1$. 

Now suppose that the vertex $n+1$ is red. Then $i_p<n+1$ and it is blue. This is illustrated in Fig.\ref{fig: step1} for $n+1=4$ and $i_p=3$.  The arc $a_{i_p}$ decomposes the Möbius strip into an integral strip whose $x$-vertices are $n+1, 1,\ldots, i_p$ and $y$-vertices are $i_p, i_p+1,\ldots, n+1$. None of the 0-simplices of $\yacross{n+1}$ corresponds to an arc that has an endpoint on the $y$-vertex in $\{i_p+1,\ldots, i_p\}$ because all of them are rejected arcs, by the assumption that $i_p<n+1$. Then any arc that is disjoint from ${(i_p, n+1)}$ is disjoint from the arc $(1, i_p)$. So $a_{i_p}$ is vertex dominated by $(1,i_p)$. This concludes the base step.

Let $i_b\in \intbra[1,i_p]$ be the smallest integer such that the $i_b$-th vertex of $\mob {n+1}$ is blue. In Fig. \ref{fig: step1}, we have $i_b=1$. 
\end{step}
\begin{step} \label{inductionstep}
	Suppose that for $i'=n+1, n, \ldots, i+1$, the 0-simplex $a_{i'}$ is vertex-dominated in $\across{n+1}   \smallsetminus \{a_{n+1}, a_n, \ldots, a_{i'+1} \}.$ Then we need to show that the 0-simplex $[a_{i}]$ is vertex-dominated in $\across{n+1}   \smallsetminus \{a_{n+1}, a_n, \ldots, a_{i+1} \}.$ The arc $a_{i}$ decomposes the Möbius strip into an integral strip whose $x$-vertices are $n+1, 1,\ldots, i$ and $y$-vertices are $i, i+1,\ldots, n+1$. By induction hypothesis, none of the 0-simplices of $ \across{n+1}   \smallsetminus \{a_{n+1}, a_n, \ldots, a_{i} \}$ corresponds to an arc that has an endpoint on the $y$-vertex labeled $n+1$. So we have that
$$\Link{a_{i}  }{ \across{n+1}   \smallsetminus \{a_{n+1}, a_n, \ldots, a_{i+1} \}   } = \acs {i,n-i+2}\Join [(1, i)].$$ So $a_{i}$ is vertex-dominated by $[(1, i)]$. This finishes the induction on $i'$. By induction hypothesis on $n$, the complex $\across{n}$ is strongly collapsible, which implies that the complex $\across{n+1}$ is strongly collapsible. This finishes the induction on $n$. 
	\end{step}

\end{proof}

Then, we show that the full arc complex non-orientable crown is collapsible.
\begin{theorem}
For all $n\geq 1$, the simplicial complex $\ac{\mob n}$ simplicially collapses onto $\across n $.
\end{theorem}
The proof is identical to the one in \cite{ppstrong} in the undecorated case.
\newpage 

\section{Three-dimensional example: the decorated 2-crown} \label{sec:app2}

In this final appendix we illustrate Theorem~\ref{thm:mainB} and Proposition~\ref{prop:dimofmac} in the case of a fully decorated 2-crown $\dholed 2$. 
This case was chosen because the corresponding arc complex $\mathcal{A}(\dholed 2)$ is 3-dimensional. 
Since the surface is fully decorated, its admissible cone is properly convex, i.e.\ bounded in an affine chart, which also makes it easier to draw realistically: it turns out to be a triangular prism.

Figure~\ref{fig:admdeco2gon} shows: 
\begin{itemize}
\item (Left panel): the surface $\dholed 2$, with some arcs color-coded by their orbits under the mapping class group $(\mathbb{Z}/2\mathbb{Z})^2$. There are 5 orbits.
\item (Middle panel): the associated bicolored disk and its diagonals, according to the conventions used in the proof of Proposition~\ref{prop:dimofmac}.
\item (Right panel): the projectivized admissible cone $\overline{\Lambda(m)}^+$, equal to the image of the arc complex $\mathcal{A}(\dholed 2)$ under the strip map $f$, where $m$ is a hyperbolic metric on the decorated 2-crown $\dholed 2$. Vertex colors match arcs of the left panel.
\end{itemize}

Figures~\ref{fig:facets}-\ref{fig:codim2}-\ref{fig:vertices} show respectively the faces, edges and vertices of $\overline{\Lambda(m)}^+$, using a similar convention, displaying the associated spread subsets in red (with blue interior, where applicable) in the left columns. 
The complexes in the right columns are all subcomplexes of the right panel of Figure~\ref{fig:admdeco2gon}, and we respect the color coding for the vertices.

\begin{figure}[H]
    \centering
    \includegraphics[width=\linewidth] {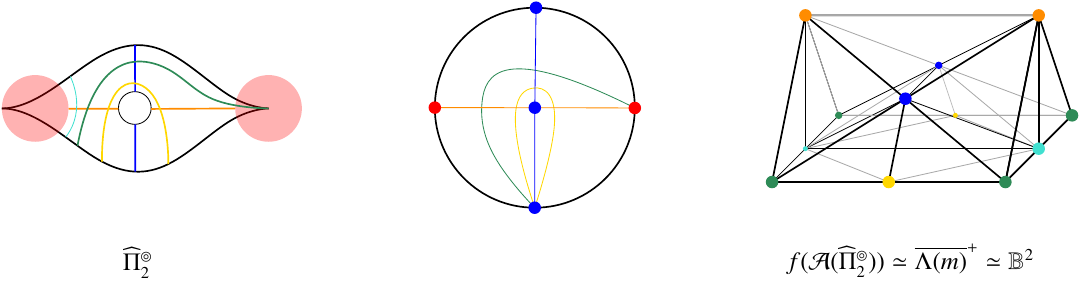}
    \caption{The admissible cone of $\dholed 2$.}
    \label{fig:admdeco2gon}
\end{figure}
\begin{figure}[H]
    \centering
    \includegraphics[height=0.8\textheight] {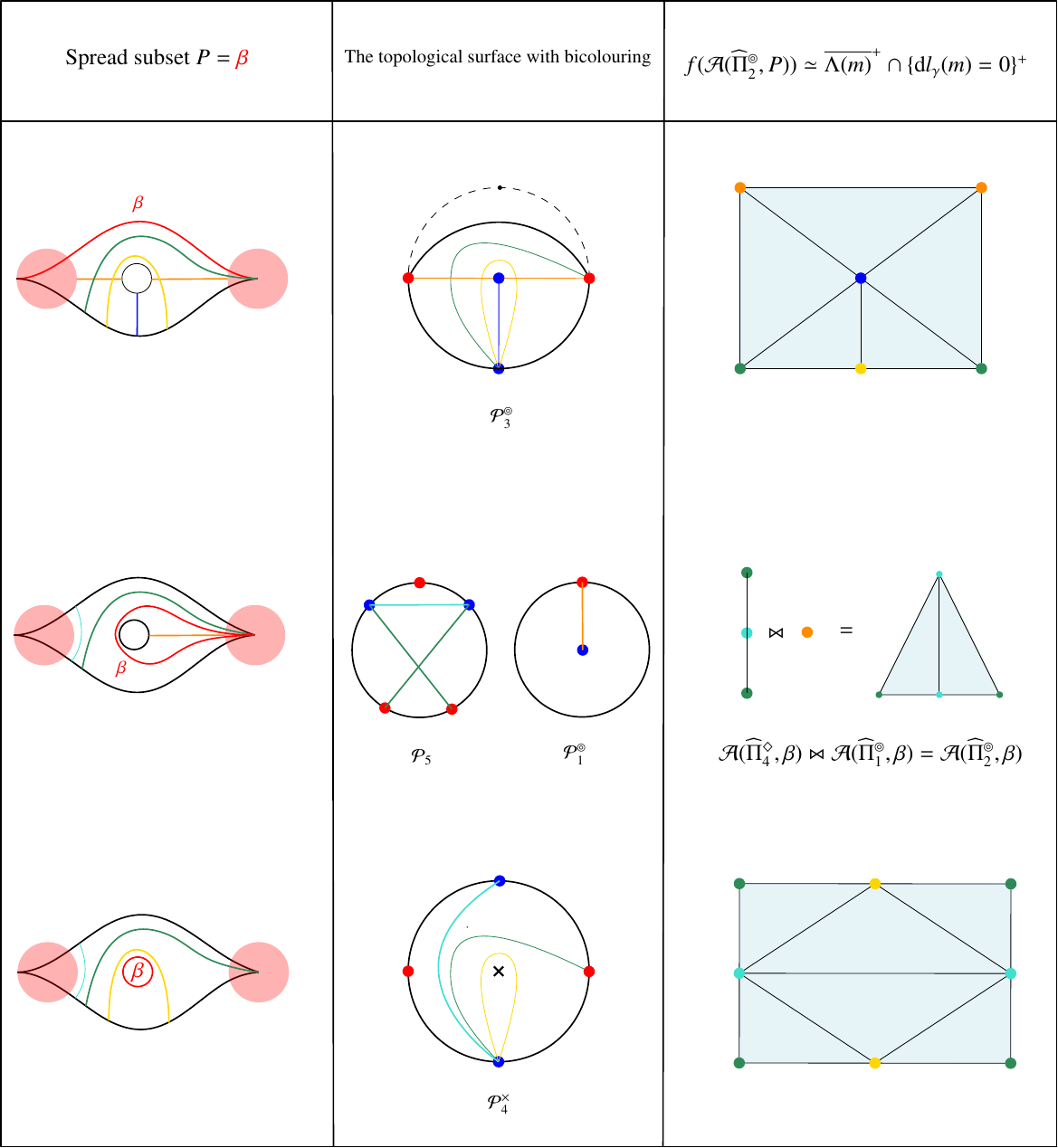}
    \caption{The facets of the admissible cone of $\dholed 2$.}
    \label{fig:facets}
\end{figure}
\begin{figure}[H]
    \centering
    \includegraphics[height=0.8\textheight] {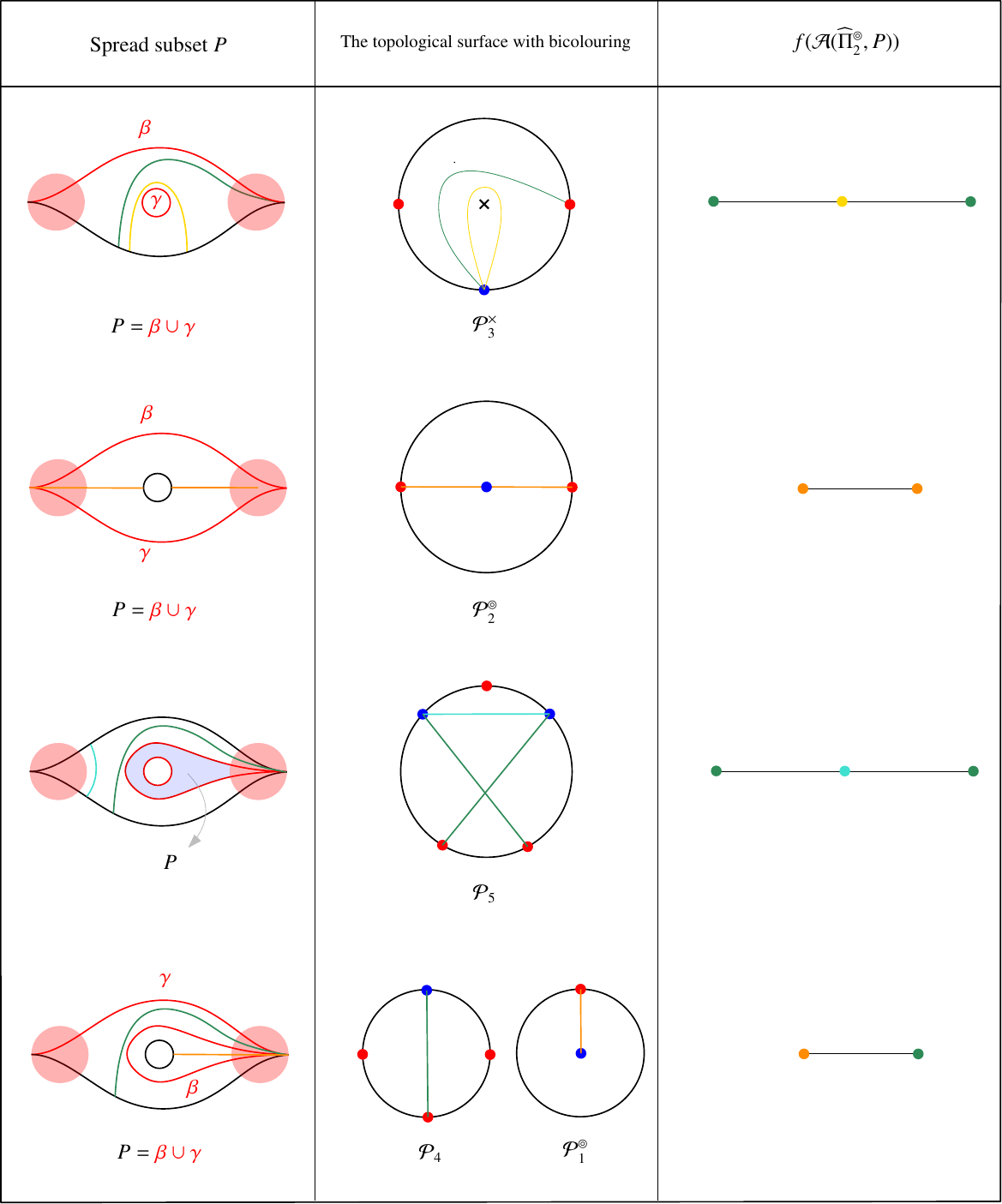}
    \caption{The codimension one faces of the admissible cone of $\dholed 2$.}
    \label{fig:codim2}
\end{figure}
\begin{figure}[H]
    \centering
    \includegraphics[width=0.7\linewidth]{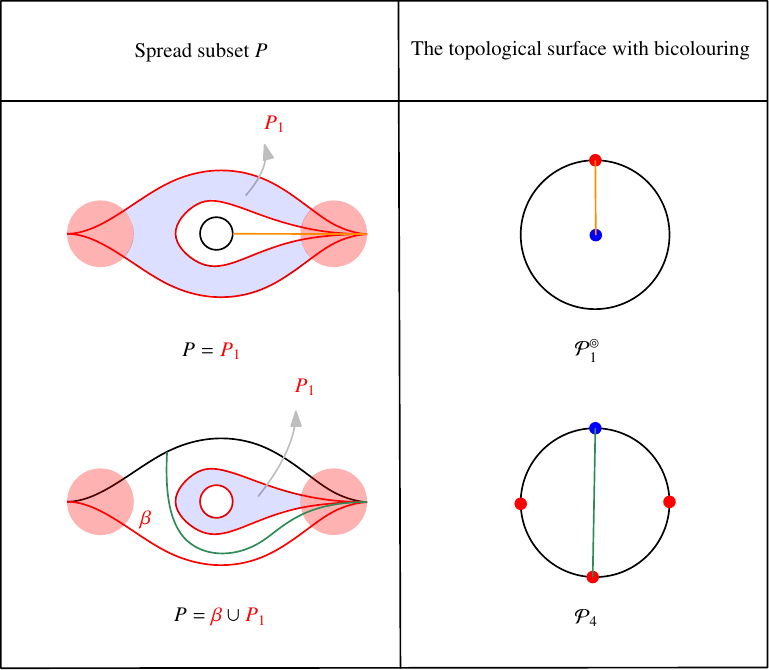}
    \caption{The vertices of the admissible cone of $\dholed 2$.}
    \label{fig:vertices}
\end{figure}
\end{appendices}

	\newpage

	\printbibliography
\end{document}